\documentclass[reqno,11pt,letterpaper]{amsart}
\usepackage{amsmath,amssymb,amsthm,graphicx,mathrsfs,url}
\usepackage[usenames,dvipsnames]{color}
\usepackage[colorlinks=true,linkcolor=Red,citecolor=Green]{hyperref}

\def\?[#1]{\textbf{[#1]}\marginpar{\Large{\textbf{??}}}}

\newtheorem{theo}{Theorem}
\newtheorem{prop}{Proposition}[section]

\newtheorem{lemm}[prop]{Lemma}
\newtheorem{corr}[prop]{Corollary}
\newtheorem{rem}{Remark}
\numberwithin{equation}{section}

\newcommand{\mc}{\mathcal}
\newcommand{\rr}{\mathbb{R}}
\newcommand{\nn}{\mathbb{N}}
\newcommand{\cc}{\mathbb{C}}
\newcommand{\hh}{\mathbb{H}}
\newcommand{\zz}{\mathbb{Z}}

\newcommand{\la}{\lambda}
\newcommand{\eps}{\epsilon}

\newcommand{\pl}{\partial}
\newcommand{\x}{\times}

\newcommand{\til}{\widetilde}
\newcommand{\bbar}{\overline}

\newcommand{\cjd}{\rangle}
\newcommand{\cjg}{\langle}

\newcommand{\ndemi}{\tfrac{n}{2}}

\DeclareMathOperator{\Op}{Op}

\DeclareMathOperator{\supp}{supp}

\DeclareMathOperator{\Tr}{Tr}

\newcommand{\nwc}{\newcommand}
\nwc{\ep}{\epsilon}
\nwc{\vareps}{\varepsilon}
\nwc{\Oph}{\operatorname{Op}_\hbar}
\nwc{\ra}{\rangle}

\nwc{\mf}{\mathbf} 
\nwc{\blds}{\boldsymbol} 
\nwc{\ml}{\mathcal} 

\nwc{\defeq}{\stackrel{\rm{def}}{=}}

\nwc{\cE}{\ml{E}}
\nwc{\cN}{\ml{N}}
\nwc{\cO}{\ml{O}}
\nwc{\cP}{\ml{P}}
\nwc{\cU}{\ml{U}}
\nwc{\cV}{\ml{V}}
\nwc{\cW}{\ml{W}}
\nwc{\tU}{\widetilde{U}}
\nwc{\IN}{\mathbb{N}}
\nwc{\IR}{\mathbb{R}}
\nwc{\IZ}{\mathbb{Z}}
\nwc{\IC}{\mathbb{C}}
\nwc{\IT}{\mathbb{T}}
\nwc{\IS}{\mathbb{S}}
\nwc{\tP}{\widetilde{P}}
\nwc{\tPi}{\widetilde{\Pi}}
\nwc{\tV}{\widetilde{V}}
\nwc{\rest}{\restriction}

\title[The Fried conjecture in small dimensions]{The Fried conjecture in small dimensions}

\author[N.V.~Dang]{Nguyen Viet Dang}

\address{Institut Camille Jordan (U.M.R. CNRS 5208), Universit\'e Claude Bernard Lyon 1, B\^atiment Braconnier, 43, boulevard du 11 novembre 1918, 
69622 Villeurbanne Cedex }

\email{dang@math.univ-lyon1.fr}

\author[C.~Guillarmou]{Colin Guillarmou}

\address{CNRS, Universit\'e Paris-Sud, D\'epartement de Math\'ematiques, 91400
Orsay, France}
\email{cguillar@math.cnrs.fr}

\author[G.~Rivi\`ere]{Gabriel Rivi\`ere}

\address{Laboratoire Paul Painlev\'e (U.M.R. CNRS 8524), D\'epartement de math\'ematiques, Facult\'e des sciences et technologies, Universit\'e de Lille , 59655 Villeneuve d'Ascq Cedex, France}

\email{gabriel.riviere@math.univ-lille1.fr}

\author[S.~Shen]{Shu Shen}

\address{Institut de math\'ematiques de Jussieu, Sorbonne Universit\'e, Case Courrier 247, 4 place Jussieu, 75252 Paris Cedex 05, France}

\email{shu.shen@imj-prg.fr}

\begin{document}
\begin{abstract}  
We study the twisted Ruelle zeta function $\zeta_X(s)$ for smooth Anosov vector fields $X$ acting on flat vector bundles over smooth compact manifolds. 
In dimension $3$, we prove the Fried conjecture, relating Reidemeister torsion and $\zeta_X(0)$. In higher dimensions, we show more generally that $\zeta_X(0)$  is locally constant with respect to the vector field $X$ under a spectral condition.
As a consequence, we also show the Fried conjecture for Anosov flows near the geodesic flow on the unit tangent bundle of hyperbolic $3$-manifolds. 
This gives the first examples of non-analytic Anosov flows and geodesic flows in 
variable negative curvature where the Fried conjecture holds true.
\end{abstract}

\maketitle

\section{Introduction}

Let $\mc{M}$ be a smooth ($\ml{C}^{\infty}$), compact, connected and 
oriented manifold of dimension $n$ and $E\rightarrow\ml{M}$ a smooth 
  Hermitian vector bundle with fibers $\cc^r$ equipped with a flat connection $\nabla$. 
Parallel transport via $\nabla$ induces a conjugacy class of representation $\rho: \pi_1(\mc{M})\to {\rm GL}(\cc^r)$, 
which is unitary as soon as $\nabla$ preserves $\cjg \cdot,\cdot\cjd_E$. One can then define a twisted de Rham complex on the space  $\Omega(\mc{M};E)$ of smooth twisted forms with twisted exterior derivative $d^\nabla$, 
and we denote by $H^k(\mc{M};\rho)$ its cohomology of degree $k$. We say that the complex (or $\rho$) is acyclic if $H^k(\mc{M};\rho)=0$ for each $k$. If 
$\rho$ is acyclic and unitary, Ray and Singer introduced a secondary invariant which is defined by the value at $0$ of the derivative of the spectral zeta function of the Laplacian~\cite{RaSi71}. 
They showed that this quantity $\tau_\rho(\mc{M})$ is in fact independent of the choice of the metric used to define the Laplacian, 
thus an invariant of the flat bundle. 
This is the so-called \emph{analytic torsion} and it was conjectured by Ray and Singer to be equal to the 
\emph{Reidemeister torsion}
~\cite{Re35, Fra, DR}. This conjecture was proved 
independently by Cheeger~\cite{Ch} and M\"uller~\cite{Mu} and it was 
extended to unimodular flat vector bundles by M\"uller~\cite{Muller2} 
and to arbitrary  flat vector bundles by Bismut and Zhang~\cite{BisZhang92}. For an introduction to the different notions of torsion, we refer the reader to~\cite{Mn}.


In the context of hyperbolic dynamical systems, Fried conjectured and proved in certain cases that 
the analytic torsion can in fact be related to the value at $0$ of a certain dynamical zeta function~\cite{F87} 
that we will now define. Given a (primitive) closed hyperbolic orbit $\gamma$ of a smooth vector field $X$, one can define its orientation index $\varepsilon_{\gamma}$ to be equal to $1$ when 
its unstable bundle $E_u(\gamma)$ is orientable and to $-1$ otherwise. If now $X$ is a smooth Anosov vector field on $\mc{M}$, we can define the \emph{Ruelle zeta function twisted} 
by the representation $\rho$ as~: 
\begin{equation}\label{Ruellezeta}
\zeta_{X,\rho}(\la):=\prod_{\gamma\in \mc{P}}\det(1-\varepsilon_{\gamma}\rho([\gamma])e^{-\la \ell(\gamma)}), \quad {\rm Re}(\la)>C
\end{equation}
where $\mc{P}$ denotes the set of primitive closed orbits of $X$ and $\ell(\gamma)$ the corresponding periods. 
Here $C>0$ is some large enough constant depending on $X$ and $\rho$. 
If $\rho$ is unitary and acyclic and if $X$ is the geodesic vector field on the unit tangent bundle $\ml{M}=SM$ of a hyperbolic manifold $M$, 
Fried showed that $\zeta_{X,\rho}(\la)$ extends meromorphically to $\la\in \cc$ using Selberg trace formula~\cite{FriedAnnENS} and the work or Ruelle~\cite{Rue}. 
Then he proved~\cite{FrInv} the remarkable formula (with $\dim(\ml{M})=2n_0+1$)~:
\begin{equation}\label{friedhyp}
|\zeta_{X,\rho}(0)^{(-1)^{n_0}}|=\tau_\rho(\ml{M}),
\end{equation}
where $\rho$ is the lift to $\pi_1(\ml{M})$ of an acyclic and unitary representation $\rho_0:\pi_1(M)\rightarrow U(\cc^r)$. Fried 
interpreted this formula as an analogue of the Lefschetz fixed point formula answering his own question in the case of geodesic flows~\cite[p.~441]{Fr0}~: 
\emph{is there a general connection between the analytic torsion of Ray and Singer and closed orbits of some flow (e.g. geodesic flow) ?}
 He then extended this formula~\cite{F87, Friedanalytic} 
 to various families of flows including Morse-Smale flows and conjectured in \cite[p. 66]{F87} that formula~\eqref{friedhyp} holds for compact locally homogeneous Riemannian spaces and acyclic bundles. This conjecture was proved for
non-positively curved locally symmetric spaces by Moscovici-Stanton \cite{MoSt} and Shen \cite{Sh}.
As for generalisations of the above results, Fried makes the following comment in~\cite[p.~66]{F87}: 
\emph{it is even conceivable that $(\varphi_t,E)$ is Lefschetz for any acyclic $E$ with a flat density and any $C^\omega$ contact flow $\varphi_t$.} 
In his 1995 article \cite[p.~181]{Friedanalytic}, Fried  conjectured that  the relation \eqref{friedhyp} holds for geodesic flows with variable negative curvature, an important case that can be seen as a part of a body of conjectures considered by Fried in special cases. 
As stated by Zworski~\cite[p.~5]{ZwBAMS} in the survey article: \emph{in the case of smooth manifolds of variable negative curvature, (\ref{friedhyp}) remains completely
open}.

 For analytic Anosov flows, generalizing earlier works of Ruelle~\cite{Rue},
Rugh showed in \cite{Rughanalytic} that $\zeta_{X,\rho}$ has meromorphic 
continuation to the whole complex plane when $\text{dim}(\ml{M})=3$. This was later extended to higher dimensions by Fried~\cite{Friedanalytic}. 
Then, Sanchez-Morgado \cite{Morg93,Morg96} 
proved that~\eqref{friedhyp} holds for
transitive analytic Anosov flows 
in dimension $3$ if there exists a closed orbit $\gamma$ such that, for each $j\in\{0,1\}$, $\ker( \rho([\gamma])-\varepsilon_{\gamma}^j{\rm Id})=0$ -- see also~\cite{F87} for related assumptions in 
the case of Morse-Smale flows. More recently, the meromorphic continuation of Ruelle zeta functions was proved in the case of hyperbolic dynamical systems 
with less regularity (say $\ml{C}^{\infty}$). The case of Anosov diffeomorphisms was handled by Liverani~\cite{Li05} while the case of Axiom A diffeomorphisms was treated by Kitaev~\cite{Ki} and Baladi-Tsujii~\cite{BaTs08}. Afterwards, Giulietti, Liverani and Pollicott proved that the meromorphic continuation of $\zeta_{X,\rho}$ holds for \emph{smooth} Anosov flows~\cite{GLP} --  see Remark~\ref{r:zeta-extension} below for the extension to general representations of the fundamental group. An alternative proof of this latter fact was given by Dyatlov-Zworski~\cite{DyZw13} via microlocal techniques, and extended by Dyatlov-Guillarmou \cite{DyGu1,DyGu2} to Axiom A cases. 
In the case of smooth \emph{contact} Anosov vector fields in dimension $3$ and of 
the trivial representation $1:[\gamma]\in\pi_1(\mc{M})\to 1\in\cc^*$, Dyatlov-Zworski \cite{DyZw} subsequently proved that the vanishing order of 
$\zeta_{X,1}(\lambda)$ at $0$ is 
$\lambda^{b_1(\ml{M})-2}$~\cite{DyZw} where $b_1(\ml{M})$ is the first Betti number of $\ml{M}$ -- see also~\cite{Ha18} in the case with boundary. 
Recent account about these advances can be found in~\cite{Go15, ZwBAMS}. We also refer to the 
book of Baladi~\cite{Ba} for a complete introduction to the spectral analysis of zeta functions in the case of diffeomorphisms. 
Building on these recent results in the smooth case, the purpose of this work is to bring new insights on Fried's questions regarding the links between Ruelle zeta functions and analytic torsion.

 
 \section{Statement of the main results}

Our first result answers Fried's question in dimension $3$ for smooth Anosov flows.

\begin{theo}\label{t:maintheo2} 
Suppose that $\operatorname{dim}(\ml{M})=3$ and let $E$ be a smooth Hermitian vector bundle with a flat connection $\nabla$ inducing a unitary and acyclic 
representation $\rho:\pi_1(\mc{M})\to U(\cc^r)$.
Let $X_0$ be a smooth Anosov vector field preserving a smooth volume form.
Then, there is a nonempty neighborhood $\mc{U}(X_0)\subset C^\infty(\mc{M};T\mc{M})$ of $X_0$ so that 
\[ \forall X\in \mc{U}(X_0), \quad \zeta_{X,\rho}(0)=\zeta_{X_0,\rho}(0)\neq 0.\]
In addition, if $b_1(\ml{M})\not=0$ or if there exists a closed orbit $\gamma$ of $X_0$ such that, for each 
$j\in\{0,1\}$, $\ker( \rho([\gamma])-\varepsilon_{\gamma}^j{\rm Id})=0$, then $\vert\zeta_{X,\rho}(0)\vert^{-1}=\tau_\rho(\mc{M})$ is the Reidemeister torsion for each $X\in \mc{U}(X_0)$.
\end{theo}

The second part of the Theorem is based on the approximation of smooth volume preserving Anosov flows by analytic transitive Anosov flows and it crucially uses the result of Sanchez-Morgado~\cite{Morg96} who proved the conjecture for transitive and analytic Anosov flows in the 3-dimensional case. This is also the reason of the holonomy assumption $\ker( \rho([\gamma])-\varepsilon_{\gamma}^j{\rm Id})=0$ for some 
$\gamma$, which was necessary in the argument of \cite{Morg96} as well as in other cases already treated by Fried, see e.g.~\cite[Th.~3.1, Th.~6.1]{F87}. When $b_1(\mc{M})\not=0$, this holonomy assumption can be removed by some approximation argument -- see p.~28.
Our proof of the first part of this Theorem \ref{t:maintheo2} is independent of these earlier works and it follows from a variation formula for $\zeta_{X,\rho}(0)$ with respect to $X$ which shows that $X\mapsto \zeta_{X,\rho}(0)$ 
is locally constant for unitary and acyclic representations in dimension $3$. Observe that a vector field in $\ml{U}(X_0)$ may not preserve a smooth volume form even if $X_0$ 
does. This variation property of the Ruelle zeta function at $0$ is in fact our main result and it holds more generally for smooth Anosov vector fields in any dimension 
under a certain non-resonance at $\la=0$ assumption. In order to state it, we need to recall the notion of \emph{Pollicott-Ruelle resonances}.

Given a vector field $X_0$ and connection $\nabla$, one can define the Lie derivative $\mathbf{X}_0:=d^{\nabla}\iota_{X_0}+\iota_{X_0}d^{\nabla}$ acting on smooth differential forms 
$\Omega(\ml{M};E)$. Then, one can find some $C>0$ depending on $X_0$ and $\rho$ such that 
$$R_{\mathbf{X}_0}(\lambda):=\int_0^{+\infty}e^{-t\lambda}e^{-t\mathbf{X}_0}dt:\Omega(\ml{M};E)\rightarrow\Omega^{\prime}(\ml{M},E)$$
is holomorphic for $\text{Re}(\lambda)>C$ where $\Omega^{\prime}(\ml{M};E)$ is the space of currents with values in $E$. 
For smooth Anosov flows, 
it was first proved by Butterley and Liverani that $R_{\mathbf{X}_0}(\lambda)$ has a meromorphic extension to the whole complex 
plane~\cite{BL07}. The 
poles of this meromorphic extension are called Pollicott-Ruelle resonances and this result was based on the construction of 
appropriate 
functional spaces for the differential operator $\mathbf{X}_0$ -- see also~\cite{BlKeLi02, GoLi08} in the case of diffeomorphisms and~\cite{L, GLP} for flows. Building on earlier works for diffeomorphisms~\cite{BaTs07, FRS}, 
Faure and Sj\"ostrand introduced microlocal methods to analyse the spectrum of Anosov flows~\cite{FaSj} and, among other things, they gave another proof of 
this result -- see also~\cite{Ts, DyZw13, FaTs}.
Using this meromorphic extension, our main result reads as
\begin{theo}\label{t:maintheo} Let $E$ be a smooth vector bundle with a flat connection $\nabla$. 
Then the set of smooth Anosov vector fields $X$ 
such that 
$0$ is not 
a pole of the meromorphic extension of
$R_{\mathbf{X}}(\lambda):\Omega(\ml{M};E)\rightarrow\Omega'(\ml{M};E)$
forms an open subset $\mathcal{U}\subset C^\infty(\ml{M},T\mc{M})$, and 
the map $X\in \mathcal{U} \longmapsto \zeta_{X,\rho}(0)$ is locally constant and nonzero. 
\end{theo}

This result is valid in any dimension and without any assumption on the fact that $\rho$ is unitary or that $X$ preserves 
some smooth volume form. Note from~\cite[Th.~2.1]{DaRi17c} that our condition on the poles of $R_{\mathbf{X}}(\lambda)$ implies that $\rho$ is acyclic. Yet, it is not clear if the converse is true even for a unitary representation and for a generic choice of vector field among Anosov vector fields. If we suppose in addition that $\ml{M}$ is $3$-dimensional, that $\rho$ is unitary and that $X$ preserves a smooth volume form, then we will show that the converse is indeed true and thus deduce the first part of Theorem~\ref{t:maintheo2}. This spectral assumption also implies that $\zeta_{X,\rho}(0)\neq 0$ as a consequence of~\cite{GLP, DyZw13} -- see e.g.~\cite[\S~3.1]{DyZw}. In the case of nonsingular Morse-Smale flows~\cite[Th.~3.1]{F87}, Fried proved that $\zeta_{X,\rho}(0)$ is equal to the Reidemeister torsion under certain assumptions on the eigenvalues of $\rho([\gamma])$ for every closed orbit. This geometric condition was in fact shown to be equivalent to the spectral condition we have here~\cite[\S~2.6]{DaRi17c}.

Observe now that Theorem~\ref{t:maintheo} 
says that the Ruelle zeta function evaluated at $\lambda=0$ is locally constant under a certain spectral assumption. This result 
suggests that this value should be an invariant of the acyclic representation class $[\rho]$ but it does not say a priori that it should be 
equal to the Reidemeister torsion. In dimension~$3$, this is indeed the case under the extra 
assumptions that $X_0$ preserves a smooth volume form and that $\rho$ is unitary as shown by Theorem~\ref{t:maintheo2}. For contact Anosov flows and unitary representation $\rho$, 
we prove that it is enough (in order to apply Theorem~\ref{t:maintheo}) to verify that $0$ is not 
a pole of the meromorphic extension of $R_{\mathbf{X}_0}(\lambda)$ restricted to $\Omega^{n_0}(\ml{M},E)$ where $\text{dim}(\ml{M})=2n_0+1$. 
For hyperbolic manifolds, using a factorisation of dynamical zeta functions associated to ${\bf X}$ in terms of infinite products of Selberg zeta functions associated to certain irreducible representations of ${\rm SO}(n_0)$, we can show that ${\bf X}$ has no $0$ resonance in the acyclic case when 
$\dim\mc{M}=5$ (see Proposition \ref{dim3}) and we  deduce the following extension of Fried conjecture ~\eqref{friedhyp}:
\begin{theo}\label{t:maintheo3} Suppose that $M=\Gamma\backslash 
	\hh^{3}$ is a compact oriented hyperbolic manifold of dimension $3$ and denote by $X_0$ the geodesic vector field on $\ml{M}=SM$. 
Let $E$ be a smooth Hermitian vector bundle with a flat connection $\nabla$ on $M$ inducing an acyclic and unitary 
representation $\rho: \pi_1(M)\to U(\cc^r)$.
Then, ${\bf X}_0$ has no resonance at $0$ and there exists a nonempty neighborhood $\ml{U}(X_0)\subset C^\infty(\mc{M};T\mc{M})$ of $X_0$ so that\footnote{Recall from~\cite{F87} that $\tau_{\rho}(M)^2=\tau_{\tilde{\rho}}(\ml{M})$.}
$$\forall X\in\ml{U}(X_0),\quad \zeta_{X,\tilde{\rho}}(0)=\tau_{\rho}(M)^2,$$
where $\tilde{\rho}$ is the lift of $\rho$ to $\ml{M}$.
\end{theo}
We emphasize that these are the first examples of geodesic flows with variable negative curvature where the Fried conjecture holds in dimension $n=2n_0+1>3$, except for locally symmetric spaces. In dimension $n_0>2$, the computations for the order of $0$ as a resonance of ${\bf X}_0$ on $S(\Gamma\backslash\hh^{n_0+1})$ are involved and do not always seem to be topological (cf Remark \ref{remhyp}). 

\subsection*{Organisation of the article} In section~\ref{s:preliminaries}, we describe in detail the dynamical 
framework and construct the escape function needed to build appropriate functional spaces.
In sections~\ref{s:ruelle-torsion} and~\ref{s:traceflat}, we describe the variation of the Ruelle zeta function for $\text{Re}(z)$ large.
In section~\ref{s:microlocal}, we show the analytic continuation
of our variation formula up to $z=0$ relying on the microlocal methods
of~\cite{FaSj,DyZw13}. 
In section~\ref{s:example}, we use the variation formula and methods of~\cite{Morg96, DFG, DyZw, DaRi17c} to discuss the Fried conjecture.
Finally, appendix~\ref{a:escape} gives technical details on the escape function and appendix~\ref{a:selberg} discusses Selberg's trace on symmetric tensors.

\subsection*{Conventions} For a smooth compact manifold $\mc{M}$, we will always use the following terminology: $T^*_0\mc{M}:=\{(x,\xi)\in T^*\mc{M};\xi\not=0\}$, $\mc{D}'(\mc{M})$ is the space of distributions, (i.e. the dual to the space of smooth functions, once a fixed smooth density has been chosen), $H^s(\mc{M}):=
(1+\Delta)^{-s/2}L^2(\mc{M})$ if $\Delta$ is the Laplacian of some fixed 
Riemannian metric on $\mc{M}$. 
If $B$ is a regularity space (such as $C^k,H^s,C^\infty,\mc{D}'$) and $E$ a smooth 
vector bundle on $\mc{M}$, $B(\mc{M};E)$ 
denotes the space of sections with regularity $B$. 
A set $\Gamma\subset T^*\mc{M}$ (or $\subset T_0^*M$) is called conic if 
$(x,\xi)\in \Gamma$ implies $(x,t\xi)\in \Gamma$ for all $t>0$.

\subsection*{Acknowledgements} 
We would like to thank V. Baladi for pointing out to us the papers \cite{Morg93,Morg96} which are used in the last part of Theorem \ref{t:maintheo2}. 
We also would like to thank Y. Bonthonneau, N.T. Dang, P. Dehornoy, F. Faure, S. Gou\"ezel, B. Hasselblatt, B. Kuester, F. Naud, H. H. Rugh, H. Sanchez-Morgado, 
T. Weich for discussions, answers to our questions and crucial remarks on this project. We also thank the referees for their detailed and useful comments that helped us to improve the presentation of our proofs. This project has received funding from the European Research Council (ERC) under the European Union’s Horizon 2020 research and innovation programme (grant agreement No. 725967). 
CG and GR were  partially supported by the ANR project GERASIC (ANR-13-BS01-0007-01) and GR also acknowledges the support of the Labex CEMPI (ANR-11-LABX-0007-01).

\section{Dynamical and analytical preliminaries}\label{s:preliminaries}

Let $X$ be a smooth vector field on a $n$-dimensional compact manifold $\ml{M}$, and denote by $\varphi_t^X$ its flow on $\ml{M}$. Recall that a vector field is said to be \emph{Anosov} 
if there exist some constants $C,\lambda>0$ and a $d\varphi_t$-invariant continuous splitting
such that, for every $t\geq 0$,
\begin{equation}\label{e:splitting-Anosov}
\begin{gathered} 
T\ml{M}=\IR X\oplus E_u(X)\oplus E_s(X),\\
\forall v\in E_s(X,x),\ \|d\varphi_t^X(x) v\|\leq Ce^{-\lambda t}\|v\|,\quad\forall v\in E_u(X,x),\ \|d\varphi_{-t}^X(x) v\|\leq Ce^{-\lambda t}\|v\|.
\end{gathered}
 \end{equation}

Here we have equipped $\ml{M}$ with a smooth Riemannian metric $g$ that will be fixed all along the paper. The subset of Anosov vector fields
\[\mc{A}:=\{X\in C^\infty(\mc{M};T\mc{M}): X \textrm{ is Anosov}\}\]
forms an open subset of $C^\infty(\mc{M};T\mc{M})$ in the $\ml{C}^{\infty}$ topology. 
Next,  we introduce the dual decomposition to~\eqref{e:splitting-Anosov}:
\[T^*\ml{M}=E_0^*(X)\oplus E_u^*(X)\oplus E_s^*(X)\] 
where  $
E_0^*(X)\left(E_u(X)\oplus E_s(X)\right)=\{0\}$,  $E_{s/u}^*(X)\left(E_{s/u}(X)\oplus \rr X\right)=\{0\}$.
We have for every $t\geq 0$,
\begin{equation}\label{e:hyperbolicity-rate-stable}
\begin{gathered}
\forall v\in E_s^*(X,x),\ \|(d\varphi_{t}^X(x)^T)^{-1} v\|\leq Ce^{-\lambda t}\|v\|,\\
\forall v\in E_u^*(X,x),\ \|(d\varphi_{-t}^X(x)^T)^{-1} v\|\leq Ce^{-\lambda t}\|v\|,
\end{gathered}\end{equation}
where $C>0$ may be larger than in~\eqref{e:splitting-Anosov}.

We define the \emph{symplectic lift} of $\varphi_t^X$ as follows:
$$\forall(x,\xi)\in T^*\ml{M},\ \Phi^X_t(x,\xi):=\left(\varphi^X_t(x),(d\varphi^X_{t}(x)^T)^{-1}\xi\right),$$
and the induced flow on $S^*\ml{M}$:
$$\tilde{\Phi}^X_t(x,\xi):=\left(\varphi^X_t(x),\frac{(d\varphi^X_{t}(x)^T)^{-1}\xi}{\left\|(d\varphi^X_{t}(x)^T)^{-1}\xi\right\|_{\varphi^X_t(x)}}\right).$$
The flow $\Phi^X_t$ is 
the Hamiltonian flow 
corresponding to 
the Hamiltonian $H(x,\xi):=\xi(X(x))$. The vector fields corresponding 
to these lifted flows will be denoted by $X_H$ and $\tilde{X}_H$.

\subsection{Invariant neighborhoods}\label{Sec:invneigh}

Fix some $X_0\in \mc{A}$. We will now recall how to construct cones adapted to the Anosov structure. For that purpose, 
we decompose any given $\xi\in T_x^*\ml{M}$ as 
$$\xi=\xi_0+\xi_u+\xi_s\in E_0^*(X_0,x)\oplus E_u^*(X_0,x)\oplus E_s^*(X_0,x),$$
and we define a new norm on the fibers of $T^*\ml{M}$
$$\|\xi\|'_x:=\|\xi_0\|_x+\int_{-\infty}^0e^{-\lambda_0} t\|(d\varphi^{X_0}_{t}(x)^T)^{-1}\xi_u\|_{\varphi^{X_0}_{t}(x)}dt+\int_0^{+\infty}e^{\lambda_0 t}
\|(d\varphi^{X_0}_{t}(x)^T)^{-1}\xi_s\|_{\varphi^{X_0}_{t}(x)}dt,$$
with $\lambda_0>0$ small enough to ensure that the integrals converge. With these conventions, one has, for every $t_0\geq 0$,
\[\begin{gathered}
\forall\xi\in E_s^*(X_0,x),\quad\|(d\varphi^{X_0}_{t_0}(x)^T)^{-1}\xi\|'\leq e^{-\lambda_0 t_0}\|\xi\|',\\
\forall\xi\in E_u^*(X_0,x),\quad\|(d\varphi^{X_0}_{-t_0}(x)^T)^{-1}\xi\|'\leq e^{-\lambda_0 t_0}\|\xi\|'.
\end{gathered}\]
Note also that, provided the initial metric $g$ is chosen in such a way that $\|X_0(x)\|_x=1$ for every $x$ in $\ml{M}$, one has, for every $t_0\in \IR$, 
$$\forall\xi\in E_0^*(X_0,x),\quad\|(d\varphi^{X_0}_{t_0}(x)^T)^{-1}\xi\|'=\|\xi\|'.$$
In other words, we have constructed a norm adapted to the dynamics of $\varphi^{X_0}_t$. Recall that this new norm is a priori only continuous. 
Nevertheless, we may use it to define 
stable and unstable cones. We fix a small parameter $\alpha>0$ and we introduce a strongly stable cone and a weakly unstable one:
\[\begin{gathered}
C^{ss}(\alpha):=\left\{(x,\xi)\in T_0^*\ml{M}: \|\xi_u+\xi_0\|_x'\leq \alpha \|\xi_s\|_x'\right\},\\
C^{u}(\alpha):=\left\{(x,\xi)\in T_0^*\ml{M}: \alpha\|\xi_u+\xi_0\|_x'\geq  \|\xi_s\|_x'\right\}.
\end{gathered}\]
In the following, $\alpha$ is always chosen strictly less than $1$ to ensure that $C^{ss}(\alpha)\cap C^{u}(\alpha)=\emptyset$. 
We have the following properties, for every $t\geq 0$,
\[\begin{gathered}
\forall(x,\xi)\in C^{ss}(\alpha),\quad \|(d\varphi^{X_0}_{-t}(x)^T)^{-1}(\xi_u+\xi_0)\|_{\varphi_{-t}^{X_0}(x)}' \leq e^{-t\lambda_0}\alpha\|(d\varphi^{X_0}_{-t}(x)^T)^{-1}(\xi_s)\|_{\varphi_{-t}^{X_0}(x)}',\\
\forall(x,\xi)\in C^{u}(\alpha),\quad \alpha e^{-t\lambda_0}\|(d\varphi^{X_0}_{t}(x)^T)^{-1}(\xi_u+\xi_0)\|_{\varphi_{t}^{X_0}(x)}' \geq\|(d\varphi^{X_0}_{t}(x)^T)^{-1}(\xi_s)\|_{\varphi_{t}^{X_0}(x)}'.\end{gathered}\]
In particular, the cone $C^u(\alpha)$ (resp. $C^{ss}(\alpha)$) is stable under the forward (resp. backward) flow of $\varphi^{X_0}_t$. Similarly, we define two cones $C^{uu}(\alpha)$ and $C^s(\alpha)$ with $-X_0$ replacing $X_0$ in the definitions. 
The following result 
 
will be useful in our analysis:
\begin{lemm}\label{l:time-uniform} Let $X_0\in\ml{A}$ and let $0<\alpha<1$ so that $C^{ss}(\alpha)\cap C^{u}(\alpha)=\emptyset$. There exist a neighborhood $\ml{U}_{\alpha}(X_0)$ of $X_0$ in the $C^{\infty}$ topology such that
$$\forall X\in\ml{U}_{\alpha}(X_0),\ \forall t\geq 1,\ \Phi^X_{-t}(C^{ss}(\alpha))\subset C^{ss}(\alpha),\ \text{and}\ \Phi^X_{t}(C^{u}(\alpha))\subset C^{u}(\alpha).$$
\end{lemm}
\begin{proof} We only discuss the case of $C^u(\alpha)$ as the other case is similar. From the above construction, one knows that 
$$\Phi_1^{X_0}\left(C^u(\alpha)\right)\subset C^u\left(\alpha e^{-\lambda_0}\right)\Subset C^u\left(\alpha e^{-\frac{\lambda_0}{2}}\right)\Subset C^u(\alpha).$$
Hence, for any $X$ in a small neighborhood of $X_0$ in the $\ml{C}^{\infty}$-topology, one has
$$\Phi_1^{X}\left(C^u(\alpha)\right)\subset  C^u\left(\alpha e^{-\frac{\lambda_0}{2}}\right),$$
from which one infers that, for every positive integer $j$,
$\Phi_j^{X}\left(C^u(\alpha)\right)\subset  C^u\left(\alpha e^{-\frac{\lambda_0}{2}}\right).$ Now, as for every $0\leq t\leq 1$,
$$\Phi_t^{X_0}\left(C^u\left(\alpha e^{-\frac{\lambda_0}{2}}\right)\right)\subset  C^u\left(\alpha e^{-\frac{\lambda_0}{2}}\right),$$
one can deduce that, for every $X$ close enough to $X_0$, one has
$$\forall t\geq 1,\quad\Phi_t^{X}\left(C^u(\alpha)\right)\subset  C^u(\alpha).$$
This concludes the proof of the Lemma.
\end{proof}

\subsection{Escape functions}

In order to study analytical properties of Anosov flows, we  make use of the microlocal tools developped by Faure-Sj\"ostrand~\cite{FaSj}, Dyatlov-Zworski~\cite{DyZw13}. One of the key ingredients of these spectral constructions is the existence of an escape function:

\begin{lemm}[Escape functions]\label{l:escape-function} There exist a function $f\in C^{\infty}(T^*\ml{M},\IR_+)$ which is $1$-homogeneous for 
$\|\xi\|_x\geq 1$, a constant $c_0>0$ and a constant $0<\tilde{\alpha}_0<1$ such that the following properties hold:
\begin{enumerate}
 \item $f(x,\xi)=\|\xi\|_x$ for $\|\xi\|_x\geq 1$ and $(x,\xi)\notin C^{uu}(\tilde{\alpha}_0)\cup C^{ss}(\tilde{\alpha}_0)$,
 \item for every $N_1>16N_0>0$ and $0<\alpha_0<\tilde{\alpha}_0$, there exist $0<\alpha_1<\alpha_0$ and  
 a neighborhood $\ml{U}(X_0)$ of $X_0$ in the $\ml{C}^{\infty}$-topology for which one can construct, 
 for any $X$ in $\ml{U}(X_0)$, a smooth function
 $$m_X^{N_0,N_1}:T^*\ml{M}\rightarrow[-2N_0,2N_1]$$
 with the following requirements
\begin{itemize}
  \item $m_X^{N_0,N_1}$ is $0$-homogeneous for $\|\xi\|_x\geq 1$,
  \item $m_X^{N_0,N_1}(x,\xi/\|\xi\|_x)\geq N_1$ on $C^{ss}(\alpha_1)$, $m_X^{N_0,N_1}(x,\xi/\|\xi\|_x)\leq -N_0$ on $C^{uu}(\alpha_1)$, 
   \item $m_X^{N_0,N_1}(x,\xi/\|\xi\|_x)\geq \frac{N_1}{8}$ outside $C^{uu}(\alpha_0)$
  \item there exist $R\geq 1$ such that, for every $X\in\ml{U}(X_0)$ and for every $(x,\xi)$ outside a small vicinity of $E_0^*(X_0)$ (independent of $X$), 
  one has
  \begin{equation}\label{e:decay-escape-function}\|\xi\|_x\geq R\quad \Longrightarrow \quad X_H(G_X^{N_0,N_1})(x,\xi)\leq -2c_0\min\{N_0,N_1\},\end{equation}
  where 
  \begin{equation}\label{GX}
  G_X^{N_0,N_1}(x,\xi):=m_{X}^{N_0,N_1}(x,\xi)\ln(1+f(x,\xi)),
  \end{equation}
  and where $R$ can be chosen equal to $1$ on $C^{uu}(\alpha_1)\cup C^{ss}(\alpha_1)$.
  \item there exists a constant $C_{N_0,N_1}>0$ such that, for every $X\in\ml{U}(X_0)$,
  \begin{equation}\label{e:bound-derivative-escape-function}
  \|\xi\|_x\geq R\quad \Longrightarrow \quad X_H(G_X^{N_0,N_1})(x,\xi)\leq C_{N_0,N_1},\end{equation}
\end{itemize}
\item Moreover,
$$X\in C^{\infty}(\ml{M};T^*\mc{M})\rightarrow m_X^{N_0,N_1}\in\ml{C}^{\infty}(T^*\ml{M},[-2N_0,2N_1])$$
is a smooth function.
\end{enumerate}
\end{lemm}

Under this form, this Lemma was proved in~\cite[Lemma~1.2]{FaSj} (or Lemma~\cite[Lemma~C.1]{DyZw13}). For our purpose, 
the only inputs with the statements from these references is that we need the escape function to depend 
smoothly on the vector field $X$ and 
the conic neighborhoods must be chosen uniformly w.r.t. $X$. 
We postpone the proof of this Lemma to Appendix~\ref{a:escape}. 
Note that, compared with the construction of~\cite{FaSj}, we do not have decay of the escape function 
$G_X^{N_0,N_1}$ in a small vicinity of the flow direction but this will be compensated by the ellipticity of the principal 
symbols in these directions -- see e.g. the proof of Proposition~\ref{p:contofresolvent} below. 
We could have chosen $f(x,\xi)$ to depend on $X$ and in that manner, we would get 
$X_H(G_X^{N_0,N_1})\leq 0$ for every $\xi$ large enough even near the flow direction -- see~\cite{FaSj}. Despite the fact that 
$f(x,\xi)$ is not equal to $\|\xi\|_x$ in a vicinity of $E_u^*$ and of $E_s^*$, we emphasize that 
$C^{-1}\|\xi\|_x\leq f(x,\xi) \leq C\|\xi\|_x$ for $|\xi|\geq 1$ (for some uniform constant $C>0$). 

The different properties stated in this Proposition may look technical but, except for the third part of~(2), they all played a role in the microlocal proof given in~\cite{FaSj}. Even if it will not be used in our analysis, this extra point compared with~\cite{FaSj} can be used to describe the wavefront set of the resonant states uniformly for a family of Anosov vector fields near a given $X_0$. This Lemma is also used under that form in \cite{GKL}. This is the reason why we wrote the it in such a generality.

\subsection{Pollicott-Ruelle spectrum}

Consider a smooth complex vector bundle $E\to \mc{M}$ equipped with a flat connection $\nabla:
\Omega^0(\ml{M},E)\rightarrow\Omega^1(\ml{M},E)$, where we denote
$\Omega^k(\mc{M},E)=C^\infty(\ml{M}; \Lambda^k(T^*\mc{M}) \otimes E)$. This connection induces a representation 
\begin{equation}\label{represent}
\rho : \pi_1(\mc{M})\to {\rm GL}(\cc^r)
\end{equation}
by taking $\rho([\gamma])$ to be the parallel transport with respect to $\nabla$ along a representative $\gamma$ of $[\gamma]\in \pi_1(\mc{M})$. We also denote by $\ml{E}$ the graded vector 
bundle $$\ml{E}:=\bigoplus_{k=0}^n \mc{E}^k, \quad \mc{E}^k:=\wedge^k(T^*\mc{M})\otimes E.$$ Associated with this connection is a twisted exterior derivative $d^\nabla$ acting on 
the space $\Omega(\ml{M},E)=\oplus_{k=0}^n\Omega^k(\mc{M},E)$. 
Since $\nabla$ is flat, one has $d^\nabla\circ d^\nabla=0$. As before, we 
fix a smooth Riemannian metric $g$ on $\ml{M}$ and a smooth hermitian structure $\langle.,.\rangle_{E}$ on $E$. This induces a scalar product on $\Omega(\ml{M},E)$ by setting, for every 
$(\psi_1,\psi_2)\in \Omega^k(\ml{M},E)$,
\[\langle \psi_1,\psi_2\rangle_{L^2}:=\int_{\ml{M}}\langle \psi_1,\psi_2\rangle_{\mc{E}^k}{\rm dvol}_g.\]
We set $L^2(\ml{M},\ml{E})$ (or $L^2(\ml{M})$ if there is no ambiguity) to be the completion of $\Omega(\ml{M},E)$ for this scalar product. The set of
De Rham currents  
valued\footnote{Observe that $E'$ can be identified with $E$ via the Hermitian structure.} in $E$ is denoted by $\ml{D}^{\prime}(\ml{M},E)$.

Given $X\in \mc{A}$, we define the twisted Lie derivative  
\begin{equation}\label{defbfX}
\mathbf{X}:= i_{X}d^\nabla+d^\nabla i_{X}:\Omega(\ml{M},E)\rightarrow \Omega(\ml{M},E).\end{equation} 
In local coordinates $(x_1,\ldots,x_n)$, its action can be written as follows. Fix $(e_j)_{1\leq j\leq r}$ a local basis of the vector bundle $E$. For every $1\leq j\leq r$ and every $J\subset\{1,\ldots,n\}$, one has
$$\mathbf{X}\left(u_{j,J}dx_J\otimes e_j\right)=\mathcal{L}_X(u_{j,J})dx_J\otimes e_j+u_{j,J}\mathcal{L}_{X}(dx_J)\otimes e_{j}+u_{j,J}dx_J\otimes\left(\nabla_Xe_j\right)$$
where $\mathcal{L}_X$ is the standard Lie derivative acting on smooth forms and 
$dx_J=dx_{J_1}\wedge\dots,dx_{J_{k}}$ if $J=(J_1,\dots,J_k)$. Hence, as the last two terms in the above sum are of order $0$, the differential operator $-i{\bf X}$ has diagonal principal symbol given by 
\begin{equation}\label{princsymb}
\sigma(-i{\bf X})(x,\xi)=H(x,\xi){\rm Id}_{\ml{E}}
\end{equation}
(recall $H(x,\xi)=\xi(X(x))$). Note that ${\bf X}$ preserves $\Omega^k(\mc{M},E)$ for each $k$.   
Also, since $[{\bf X} ,i_{X}]=0$, it also preserves sections of the bundle (depending smoothly on $X$)
\begin{equation}\label{Omega_0}
\mc{E}_0:=\mc{E}\cap \ker i_{X}=\bigoplus_{k=0}^{n-1}\underset{:=\mc{E}_0^k}{\underbrace{\mc{E}^k\cap \ker i_{X}}}.
\end{equation}

It was shown in~\cite{BL07, FaSj, GLP, DyZw13} that this differential operator 
has a discrete spectrum when acting on convenient Banach spaces of currents. Let us recall this result using the microlocal 
framework from~\cite{FaSj, DyZw13}. Using~\cite[Th.~8.6]{Zw} and letting $N_0,N_1>0$ be two positive parameters, we set
\begin{equation}\label{e:sobolev-weight}
\mathbf{A}_h(N_0,N_1,X):=\Op_h\left(e^{G_X^{N_0,N_1}}\text{Id}_{\ml{E}}\right),
\end{equation}
where $\Op_h$ is a semiclassical quantization procedure on $\ml{M}$~\cite[Th.~14.1]{Zw}. 
As in \cite[Lemma 12]{FRS} (or the semiclassical version of it), if $h>0$ is small enough there is a semiclassical 
pseudo-differential operator $\til{\bf A}^{-1}_h(N_0,N_1,X)$ with semiclassical principal symbol $e^{-G_X^{N_0,N_1}}$ so that $\mathbf{A}_h(N_0,N_1,X)\til{\bf A}^{-1}_h(N_0,N_1,X)={\rm Id}+h^\infty\Psi^{-\infty}(\mc{M};\mc{E})$, which implies that $\mathbf{A}_h(N_0,N_1,X)^{-1}$ exists and is a pseudodifferential operator with principal symbol $e^{-G_X^{N_0,N_1}}$ and local full symbol (in charts) that only involves derivatives of $G$. 

We then define the (semiclassical) 
\emph{anisotropic Sobolev spaces}:
$$\forall 0<h\leq 1,\quad\ml{H}_h^{m^{N_0,N_1}_X}(\ml{M},\ml{E}):=\mathbf{A}_h(N_0,N_1,X)^{-1}L^2(\ml{M};\ml{E}),$$
where we used the subscript $X$ to remind the dependence of these spaces on the vector field $X$. These spaces are related to the usual semiclassical Sobolev spaces $H_h^{k}(\mc{M};\mc{E}):=(1+h^2\Delta_{\mc{E}})^{-k/2}L^2(\mc{M};\mc{E})$ as follows ($\Delta_{\mc{E}}$ is some positive Laplacian on $\mc{E}$)
\begin{equation}\label{e:sobolev-injection}
H_h^{2N_1}(\ml{M},\ml{E})\subset \ml{H}_h^{m^{N_0,N_1}_X}(\ml{M},\ml{E})\subset H_h^{-2N_0}(\ml{M},\ml{E}),
\end{equation}
with continuous injections. Stated in the case of a general smooth vector bundle $E$, the main results from~\cite[Th.~1.4-5, \S~5]{FaSj} and~\cite[Prop.~3.1-3]{DyZw13} read as follows:
\begin{prop}\label{p:ruelle-spectrum} 
Let $X$ be an element in $\ml{U}(X_0)$ where $\ml{U}(X_0)$ is the neighborhood of Lemma~\ref{l:escape-function}. Then, there is $h_0>0$ small and 
exists $C_{X}>0$ (depending continuously\footnote{Even 
if not explicitely written in~\cite{FaSj}, this observation can be deduced from paragraph~3.2 
of this reference and from Lemma~\ref{l:escape-function} above.} on $X\in\ml{A}$) such that, 
for any $N_1>16N_0>0$, the resolvent
\[(\mathbf{X}+\lambda)^{-1}=\int_0^{+\infty}e^{-t\mathbf{X}}e^{-t\lambda}dt:\ml{H}_{h_0}^{m_X^{N_0,N_1}}(\ml{M},\ml{E})\rightarrow \ml{H}_{h_0}^{m_X^{N_0,N_1}}(\ml{M},\ml{E})\]
is holomorphic in $\{\operatorname{Re}(\lambda)>C_{X}\}$ and has a meromorphic extension to
$$\left\{\operatorname{Re}(\lambda)>C_{X}-c_0N_0\right\},$$ where $c_0>0$ 
is the constant from Lemma~\ref{l:escape-function}. The poles of this meromorphic extension are called the Pollicott-Ruelle resonances and the range of the residues are the corresponding  generalised resonant states. Moreover, the poles and residues of the meromorphic extension are intrinsic and do not depend on the choice of escape function used to define the anisotropic Sobolev space.
\end{prop}

This result should be understood as follows. In these references, $(\mathbf{X}+\lambda):\ml{D}(\mathbf{X})\rightarrow \ml{H}^m_h$ is shown to be 
a family of Fredholm operators of index $0$ depending analytically on $\lambda$ in the region $\{\text{Re}(\lambda)>C_{X}-c_0N_0\}$. Then, 
the poles of the meromorphic extension are the eigenvalues of $-\mathbf{X}$ on $\ml{H}_h^{m_X^{N_0,N_1}}(\ml{M},\ml{E})$. 
We shall briefly rediscuss 
the proofs of~\cite{FaSj, DyZw13} in Proposition~\ref{p:contofresolvent} below as 
we will need to control the continuity of $(\mathbf{X}+\lambda)^{-1}$ with respect to $X\in \mc{A}$. We also refer to the recent work of Guedes-Bonthonneau for related results~\cite{Bo18}.
 
\begin{rem}
 For technical reasons appearing later in the analysis of the wave-front set of the Schwartz kernel of $({\bf X}+\la)^{-1}$, we 
use a semiclassical parameter $h$ and a semiclassical quantization, even though the operator ${\bf X}+\la$ is not semiclassical. 
For this Proposition, one just fix $h=h_0$ but some statement for $h\to 0$ will be used later on in the proof of Proposition \ref{p:boundonWF}.
\end{rem}

\begin{rem}\label{r:N0N1} In the following, we will take $N_1=20N_0$ and thus we will omit the index $N_1$ in $G_X^{N_0,N_1}$, $m_X^{N_0,N_1}$ and $\mathbf{A}_h(N_0,N_1,X)$.
 
\end{rem}

\section{Twisted Ruelle zeta function and variation formula}\label{s:ruelle-torsion}

In this section, we shall introduce the Ruelle zeta function and 
derive a formula\footnote{Similar method is also used in \cite{FRZ} 
for Selberg zeta function on surfaces of constant curvature.} for its variation with respect to the vector field $X\in \mc{A}$.
More precisely,  we consider a smooth $1$-parameter family 
$\tau \in (-1,1) \mapsto X_\tau\in \mc{A}$ on $\mc{M}$ and we fix a representation $\rho:\pi_1(\mc{M})\to {\rm GL}(\cc^r)$. 
We define the \emph{Ruelle zeta function} of $(X_\tau,\rho)$ as in ~\cite{F87} by  the converging product\footnote{As we shall 
consider families $\tau\mapsto X_{\tau}$, if no confusion is possible we will use the index (or the exponent) 
$\tau$ instead of $X_{\tau}$ in the various quantities $\varphi_t^{X_{\tau}}$, $\zeta_{X_{\tau},\rho}$, etc.}
\begin{equation}\label{defRuelle}
\zeta_{\tau,\rho}(\la):=\prod_{\gamma_\tau\in \mc{P}_\tau}\det(1-\varepsilon_{\gamma_{\tau}}\rho([\gamma_\tau])e^{-\la \ell(\gamma_\tau)})
\end{equation}
for ${\rm Re}(\la)>\Lambda_\tau$ (for some $\Lambda_\tau>0$), where $\mc{P}_\tau$ is the set of primitive periodic orbits of $X_\tau$, $[\gamma_\tau]$ represents 
the class of $\gamma_\tau$ in $\pi_1(\mc{M})$, and $\ell(\gamma_\tau)$ denotes the period of the orbit $\gamma_\tau$. Recall also that $\varepsilon_{\gamma_{\tau}}$ is the orientation index of the closed orbit. 
To justify the convergence, it suffices to combine the fact that for a fixed 
Hermitian product $\cjg\cdot,\cdot\cjd_E$ on $E$, there is $C>0$ depending only on $(\nabla,E,\cjg\cdot,\cdot\cjd_E)$  such that $||\rho([\gamma_\tau])||_{E\to E}\leq e^{C\ell(\gamma_\tau)}$, together with  Margulis bound~\cite{Ma70} on the growth of periodic orbits
 \begin{equation}\label{e:margulis}
 \left|\left\{\gamma\in \mc{P}_\tau:\ \ell(\gamma_\tau)
 \leq T\right\}\right|=\mathcal{O}\left( \frac{e^{T h_{\rm top}^\tau}}{T}\right)\quad\text{as}\quad T\rightarrow+\infty
 \end{equation}
where $h_{\rm top}^\tau$ denotes the topological entropy of the flow $\varphi_t^\tau$ of $X_\tau$ at time $t=1$ -- see also~\cite[Lemma~2.2]{DyZw13}. Recall that, for weak-mixing Anosov flows, one has in fact an equivalent.

\begin{rem}\label{r:zeta-extension} In~\cite{GLP, DyZw13}, the proof of the meromorphic continuation of the Ruelle zeta function was given under the extra assumptions that $\varepsilon_\gamma=1$ for every $\gamma\in\mathcal{P}$, and that the representation $\rho$ is trivial. The analysis of the resolvent and of the dynamical zeta functions in the case of general bundle is done in \cite[\S 5, Th.~4]{DyGu1} in a setting containing the case of Anosov flows. We thus refer to that last article for the discussion about non-trivial representations and general bundles. 
The factor $\varepsilon_\gamma$ in \eqref{defRuelle} allows to compensate the assumption about orientability of the stable/unstable bundles.
\end{rem}

\subsection{Variation of lengths of periodic orbits}

The first ingredient is the following consequence of the structural stability of Anosov flows:
\begin{lemm}\label{l:variation-length} Assume that $X_0\in\mc{A}$. There exists a neighborhood $\ml{U}(X_0)$ of $X_0$ such that
$\tau\mapsto X_\tau\in \ml{U}(X_0)$ is a smooth family of Anosov vector fields on $\mc{M}$ topologically conjugated to $X_0$. Moreover, there 
is a smooth family $\tau\mapsto h_\tau\in C^0(\mc{M},\mc{M})$ of  conjugating homeomorphisms 
defined near $\tau=0$ such that $h_{\tau}(\gamma_{0})=\gamma_\tau$ for each $\gamma_{0}\in \mc{P}_{0}$, the map
$\tau\mapsto \ell(\gamma_\tau)=\ell(h(\gamma_{0}))$ is $C^1$ near $0$ for each $\gamma_{0}\in\mc{P}_{0}$, and 
\[ \pl_\tau \ell(\gamma_\tau)=-\int_{\gamma_\tau}q_\tau\]
if $\pl_\tau X_{\tau}=q_{\tau}X_{\tau}+X_{\tau}^{\perp},$
 with $X_{\tau}^{\perp}\in C^0(\mc{M};E_u(X_{\tau})\oplus E_s(X_{\tau})).$
\end{lemm}
\begin{proof}
We consider the Anosov vector field $X_0$.
Following~\cite[App.~A]{DllMaMo86}, we introduce the space $C_{X_{0}}(\ml{M},\ml{M})$ of continuous functions $h$ from $\ml{M}$ 
to $\ml{M}$ which are $C^1$ along $X_0$. This means that, for all $x$ in $\ml{M}$, the map $t\mapsto h\circ\varphi_{X_0}^t(x)$ is $C^1$ and 
the map $x\mapsto \frac{d}{dt}\left(h\circ\varphi_{t}^{X_0}(x)\right)_{t=0}=:D_{X_0}h(x)\in T\ml{M}$ is continuous. Building on earlier arguments of 
Moser and Mather for Anosov diffeomorphisms, de la Llave, Marco and Moriyon proved the structural stability theorem of Anosov via an implicit function 
theorem~\cite[App.~A]{DllMaMo86}. 
\begin{prop}[De la Llave-Marco-Moriyon \cite{DllMaMo86}]\label{t:structural} 
With the previous conventions, there exists an open neighborhood $\ml{U}(X_0)$ of $X_0$ in $\mc{A}$ and a $\ml{C}^{\infty}$ map
\[S: X\in \ml{U}(X_0)\mapsto (h_X,\theta_X)\in C_{X_0}(\ml{M},\ml{M})\times C^0(\mc{M},\IR),\]
where $S(X_0)=({\rm Id},1)$ and
\[\pl_t(h_X(\varphi_t^{0}(x)))|_{t=0}=\theta_X(x) X(h_X(x)), \quad \forall x\in \mc{M}\] 
if $\varphi^0_t$ is the flow of $X_0$. Moreover, $h_X$ is a homeomorphism of $\mc{M}$ for each $X$.
\end{prop}
We take a connected component of the curve $X_\tau$ lying in $\mc{U}(X_0)$, which amounts to consider $X_\tau$ for $|\tau|<\delta$ with $\delta>0$ small enough.
Writing the flow of $X_\tau$ by $\varphi^\tau_t$ and $h_\tau:=h_{X_\tau}$, $\theta_\tau:=\theta_{X_\tau}$,  this result can be rewritten in an integrated version:
\[\forall x\in \ml{M},\quad h_{\tau}(\varphi_t^{0}(x))=
\varphi^\tau_{\int_0^t\theta_\tau\circ\varphi_s^{0}(x)ds}(h_\tau(x)).\]
Fix now a primitive closed orbit $\gamma_0$ of the flow $\varphi_t^0$ (with period 
$\ell(\gamma_0)$) and fix a point $x_0$ on this orbit. 
From the previous formula, one has
\[h_\tau(x_0)=\varphi^\tau_{\int_0^{\ell(\gamma_0)}\theta_\tau\circ\varphi_s^{0}(x_0)ds}\left( h_\tau(x_0)\right).\]
In particular, the period of the closed orbit for $X_\tau$ equals
\[\ell(\gamma_{\tau})=\int_{\gamma_0}\theta_{\tau}\in C^{\infty}((-\delta,\delta),\IR_+^*).\]
Let us now compute its derivative by differentiating $h_{\tau}(x_{0})=\varphi_{\ell(\gamma_{\tau})}^{\tau}(h_{\tau}(x_{0}))$ at $\tau=0$:
\begin{align}\label{ss1}
	\left(\frac{\partial h_{\tau}}{\partial \tau}(x_{0})\right)_{|\tau=0}=\frac{\partial  
	}{\partial 
	\tau}\varphi^{\tau}_{\ell(\gamma_{0})}(x_{0})_{|\tau=0}+\partial_{\tau}\ell(\gamma_{\tau})_{|\tau=0}X_{0}(x_{0})+d\varphi^{0}_{\ell(\gamma_{0})}(x_{0})\cdot \left(\frac{\partial h_{\tau}}{\partial \tau}(x_{0})\right)_{|\tau=0}.
\end{align}
Let $\beta_{x_{0}}:T_{x_{0}}\mathcal{M}\to \mathbb{R}$ be defined 
such that, if $V\in T_{x_{0}}\mathcal{M}$, then 
$V=\beta_{x_{0}}(V)X_{0}(x_{0})+V^{\bot}$ 
where $V^{\bot}\in E_{u,x_{0}}(X_{0})\oplus 
E_{s,x_{0}}(X_{0})$. Pairing \eqref{ss1} with $\beta_{x_{0}}$, we get
\begin{align}\label{ss2}
	\partial_{\tau}\ell(\gamma_{\tau})_{|\tau=0}=-\beta_{x_{0}}\left(\frac{\partial  
	}{\partial 
	\tau}\varphi^{\tau}_{\ell(\gamma_{0})}(x_{0})_{|\tau=0}\right).
\end{align}
Since $\beta_{x_{0}}$ is 
$d\varphi^{0}_{\ell(\gamma_{0})}(x_{0})$ invariant, we 
have 
\begin{multline}
	\beta_{x_{0}}\left(\frac{\partial  
	}{\partial 
	\tau}\varphi^{\tau}_{\ell(\gamma_{0})}(x_{0})_{|\tau=0}\right)=\beta_{x_{0}}\left(\big(d\varphi^{0}_{\ell(\gamma_{0})}(x_{0})\big)^{-1}\cdot\frac{\partial  
	}{\partial 
	\tau}\varphi^{\tau}_{\ell(\gamma_{0})}(x_{0})_{|\tau=0}\right)\\
	=\int_{0}^{\ell(\gamma_{0})}\frac{d}{d 
	t}\beta_{x_{0}}\left(\big(d\varphi^{0}_{t}(x_{0})\big)^{-1}\cdot\frac{\partial  
	}{\partial 
	\tau}\varphi^{\tau}_{t}(x_{0})_{|\tau=0}\right)dt.
\end{multline}
On the other hand, we have
$$\frac{\partial}{\partial t}	\left(\big(d\varphi^{0}_{t}(x_{0})\big)^{-1}\cdot\frac{\partial  
	}{\partial 
	\tau}\varphi^{\tau}_{t}(x_{0})_{|\tau=0}\right)
	=d\varphi_t^0(x_0)^{-1}\frac{\partial^2}{\partial s\partial\tau}
	\left(\varphi^0_{-s}\circ\varphi_{t+s}^{\tau}(x_0)\right)_{|(s,\tau)=0},$$
and $\frac{\partial}{\partial s}(\varphi^0_{-s}\circ\varphi_{t+s}^{\tau}(x_0))=-X_0(\varphi^0_{-s}\circ\varphi_{t+s}^{\tau}(x_0))+
X_{\tau}(\varphi_{-s}^0\circ\varphi_{t+s}^{\tau}(x_0))+\ml{O}(s).$
Hence, one finds
\begin{align}\label{ss3}
\frac{\partial}{\partial t}	\left(\big(d\varphi^{0}_{t}(x_{0})\big)^{-1}\cdot\frac{\partial  
	}{\partial 
	\tau}\varphi^{\tau}_{t}(x_{0})_{|\tau=0}\right)=\big(d\varphi^{0}_{t}(x_{0})\big)^{-1}\cdot\left(\frac{\partial  X_\tau
	}{\partial 
	\tau}\big(\varphi^{0}_{t}(x_{0})\big)\right)_{|\tau=0}. 
\end{align}
By \eqref{ss2}-\eqref{ss3} and by the invariance of the Anosov splitting, 
we get the desired equation (the same argument works at each $\tau$ instead of 
$\tau=0$). 
\end{proof}
\begin{rem}\label{r:variation-period} A consequence Lemma \ref{l:variation-length}  is that, for every $\gamma_{0}\in \mc{P}_{0}$, one has
$$ \frac{\ell(\gamma_0)}{2}\leq \ell(\gamma_\tau)\leq 2\ell(\gamma_0),$$
provided that $\ml{U}(X_0)$ is chosen small enough (independently of the closed orbit).
\end{rem}

\subsection{Variation of Ruelle zeta function in the convergence region}

We start with the following result which is a consequence of Lemma~\ref{l:variation-length}.
\begin{lemm}\label{l:regularity-physical-region} Under the above assumptions, there exist $\tau_0>0$ and 
$C_0>0$ such that $X_{\tau}\in\ml{U}(X_0)$ for every $\tau\in(-\tau_0,\tau_0)$ and such that 
the map
\[\tau\in(-\tau_0,\tau_0)\mapsto \zeta_{\tau,\rho}(.)\in {\rm Hol}(\Omega_0)\]
 is of class $C^1$ where $\Omega_0:=\{\operatorname{Re}(\la)>C_0\}.$ Moreover, for every $\tau\in(-\tau_0,\tau_0)$
 $$\zeta_{\tau,\rho}(\lambda)=\zeta_{0,\rho}(\lambda)\exp\left(-\lambda\int_0^{\tau}\sum_{\gamma_{\tau'}}\frac{\ell^{\sharp}(\gamma_{\tau'})}{\ell(\gamma_{\tau'})}
 \left(\int_{\gamma_{\tau'}}q_{\tau'}\right)e^{-\lambda\ell(\gamma_{\tau'})}\varepsilon_{\gamma_{\tau'}}\operatorname{Tr}(\rho([\gamma_{\tau'}]))d\tau'\right),$$
 where the sum runs over all closed orbits of $X_{\tau'}$, $\ell^{\sharp}(\gamma_{\tau'})$ is the period of the primitive orbit generating $\gamma_{\tau'}$, $\varepsilon_{\gamma_{\tau'}}$ is the orientation 
 index\footnote{For a nonprimitive orbit $k.\gamma$, this is equal to $\varepsilon_{k.\gamma}=\varepsilon_{\gamma}^k$.} of $\gamma_{\tau'}$ and
\[\int_{\gamma_{\tau'}}q_{\tau'}=\int_0^{\ell(\gamma_{\tau'})}q_{\tau'}\circ\varphi_{t}^{\tau'}dt.\]
\end{lemm}
\begin{proof} The fact that $\lambda\mapsto\zeta_{\tau,\rho}(\lambda)$ is holomorphic in some half plane 
$\{\operatorname{Re}(\lambda)>C_{\tau}\}$ was already discussed. The fact that $C_0$ can be chosen uniformly in $\tau$ follows from Lemma~\ref{l:variation-length} and Remark~\ref{r:variation-period} together with \eqref{e:margulis} at $\tau=0$.
 Let us now compute the derivative with respect to the parameter $\tau$. For that purpose, we compute the derivative of each term in the sum defining 
$\log\zeta_{\tau,\rho}(.)$. Precisely, we write
\[\pl_\tau \left(\log\text{det}\left(\text{Id}-\varepsilon_{\gamma_\tau}e^{-\lambda\ell(\gamma_\tau)}\rho([\gamma_\tau])\right)\right) =\lambda 
\pl_\tau \ell(\gamma_\tau)\sum_{k=1}^{+\infty}e^{-k\lambda\ell(\gamma_\tau)}\varepsilon_{\gamma_\tau}^k\text{Tr}(\rho([\gamma_\tau])^k).
\]
The same kind of considerations as above allows to verify that the sum of this quantity over all primitive orbits is a 
continuous map from $(-\tau_0,\tau_0)$ to ${\rm Hol}(\Omega_0)$. Hence, 
the map $\tau\in(-\tau_0,\tau_0)\mapsto \ln\zeta_{\tau,\rho}(.)\in{\rm Hol}(\Omega_0)$ is $C^1$ with a derivative given by
\[\pl_\tau  \log\zeta_{\tau,\rho}(\lambda)=\lambda\sum_{\gamma\in \mc{P}_\tau}
\pl_\tau \ell(\gamma_\tau)\sum_{k=1}^{+\infty}e^{-\lambda k\ell(\gamma_\tau)}\varepsilon_{\gamma_\tau}^k\text{Tr}(\rho([\gamma_\tau])^k).\]
It remains to integrate this expression between $0$ and $\tau$ and use Lemma~\ref{l:variation-length}.
\end{proof}

One of the technical issue with the formula of Lemma~\ref{l:regularity-physical-region} is that $q_{\tau}$ is in general $C^0$ (or H\"older), and it makes it difficult to relate it with distributional traces as in~\cite{GLP, DyZw13}. To bypass this problem we introduce an invertible smooth bundle map $S_{\tau}: T\mc{M}\rightarrow T\mc{M}$ such that $S_\tau(X_0)=X_\tau$ and
\begin{equation}\label{Atauk}
\forall 0\leq k\leq n,\quad A_{\tau}^{(k)}:=\pl_\tau(\wedge^k 
S_{\tau})\,\left(\wedge^k S_{\tau}^{-1}\right):\wedge^k(T\mc{M})\rightarrow \wedge^k(T\mc{M}).
\end{equation}
\begin{rem} In order to get an intuition on the introduction of this extra-operator, let us observe that $S_{\tau}$ will play in some sense the role of the Hodge-star map in the analytical definition by Ray and Singer of torsion~\cite{RaSi71}. Indeed, one has $\mathbf{X}_{\tau}:=\iota_{X_{\tau}}d^{\nabla}+d^{\nabla}\iota_{X_{\tau}},$ with $\iota_{X_{\tau}}$ which can be rewritten as $\iota_{X_{\tau}}=\left(\wedge (S_{\tau}^T)^{-1}\right)\circ\iota_X\circ\left(\wedge S_{\tau}^T\right).$ Note also that, using the conventions of Lemma~\ref{l:variation-length}, one has $A_{\tau}^{(1)}(X_\tau)=q_\tau X_\tau+X_\tau^\perp.$
\end{rem}

Our next Lemma allows to express the variation of the Ruelle zeta function in terms of this bundle map $A_\tau^{(k)}$ instead of the continuous function $q_\tau$:
\begin{lemm}\label{l:shu-trick} 
With the conventions of Lemma~\ref{l:regularity-physical-region}, one has, for every $\tau_1\in(-\tau_0,\tau_0)$, for 
every closed orbit $\gamma_{\tau_1}$ and for every $x\in\gamma_{\tau_1}$,
 $$q_{\tau_1}(x)
 =-\frac{1}{\operatorname{det}\left(\operatorname{Id}-P(\gamma_{\tau_1})\right)}
 \sum_{k=0}^n(-1)^k\operatorname{Tr}\left(A_{\tau_1}^{(k)}(x)\left(\wedge^k d\varphi^{\tau_1}_{\ell(\gamma_{\tau_1})}(x)\right)\right),$$
 where $P(\gamma_{\tau_1})=d\varphi_{\ell(\gamma_{\tau_1})}^{\tau_1}(x)|_{E_u(X_{\tau_1})\oplus E_s(X_{\tau_1})}$ is the linearized 
 Poincar\'e map at $x\in\gamma_{\tau_1}$.
\end{lemm}
\begin{proof} Fix $\tau_1$ in $(-\tau_0,\tau_0)$ and $x$ belonging to a closed orbit $\gamma_{\tau_1}$. Write
$$\frac{\text{det}\left(\text{Id}-S_{\tau}S_{\tau_1}^{-1}d\varphi^{\tau_1}_{\ell(\gamma_{\tau_1})}(x)\right)}{\text{det}\left(\text{Id}-P(\gamma_{\tau_1})\right)}
=
\frac{\text{det}\left(\text{Id}-d\varphi^{\tau_1}_{\ell(\gamma_{\tau_1})}(x)- 
(S_{\tau}-S_{\tau_1})S_{\tau_1}^{-1}d\varphi^{\tau_1}_{\ell(\gamma_{\tau_1})}(x)\right)}{\text{det}(\text{Id}-P(\gamma_{\tau_1}))}.$$
We now  differentiate this expression at $\tau=\tau_1$. We have
$$(S_{\tau}-S_{\tau_1})S_{\tau_1}^{-1}=(\tau-\tau_1)\left(\frac{d S_{\tau}}{d\tau}\right)_{|\tau=\tau_1}S_{\tau_1}^{-1}+\ml{O}((\tau-\tau_1)^2).$$
Observe now that $\left(\frac{d S_{\tau}}{d\tau}\right)_{|\tau=\tau_1}S_{\tau_1}^{-1}(X_{\tau_1})=\left(\frac{d X_{\tau}}{d\tau}\right)_{|\tau=\tau_1}$.
Hence, one finds
$$q_{\tau_1}=-\frac{d}{d\tau}
\left(\frac{\text{det}\left(\text{Id}-S_{\tau}S_{\tau_1}^{-1}d\varphi^{\tau_1}_{\ell(\gamma_{\tau_1})}(x)\right)}{\text{det}\left(\text{Id}-P(\gamma_{\tau_1})\right)}
\right)_{|\tau=\tau_1}$$
by using the decomposition $\rr X_{\tau_1}\oplus E_s(X_{\tau_1})\oplus E_u(X_{\tau_1})$.
On the other hand,
$$\text{det}\left(\text{Id}-S_{\tau}S_{\tau_1}^{-1}d\varphi^{\tau_1}_{\ell(\gamma_{\tau_1})}(x)\right)=
\sum_{k=0}^n(-1)^k\text{Tr}\left(\wedge^k\left(S_{\tau}S_{\tau_1}^{-1}d\varphi^{\tau_1}_{\ell(\gamma_{\tau_1})}(x)\right)\right).$$
Differentiating this expression at $\tau=\tau_1$, this yields
$$q_{\tau_1}=-\frac{1}{\operatorname{det}\left(\operatorname{Id}-P(\gamma_{\tau_1})\right)}\sum_{k=0}^n(-1)^k\text{Tr}\left(\frac{d}{d\tau}\left(\wedge^k\left(S_{\tau}S_{\tau_1}^{-1}d\varphi^{\tau_1}_{\ell(\gamma_{\tau_1})}(x)\right)\right)_{|\tau=\tau_1}\right),$$
from which the conclusion follows.
\end{proof}
Combining Lemma~\ref{l:regularity-physical-region} and Lemma~\ref{l:shu-trick}, we get 
\begin{corr}\label{corrolairevar}
With the conventions of Lemma~\ref{l:regularity-physical-region}, one has, for every $\tau\in(-\tau_0,\tau_0)$ and for $\lambda \in \Omega_0$ 
\[
\frac{\zeta_{\tau,\rho}(\lambda)}{\zeta_{0,\rho}(\lambda)}=\exp\left(-\lambda\int_0^{\tau}\sum_{k=0}^n(-1)^{k}\sum_{\gamma_{\tau'}}\frac{\ell^\sharp(\gamma_{\tau'})}{\ell(\gamma_{\tau'})}
 \frac{\left(\int_{\gamma_{\tau'}}{\rm Tr}\left(A_{\tau'}^{(k)}\left(\wedge^k 
 d\varphi^{\tau'}_{\ell(\gamma_{\tau'})}\right)\right)\right)}{|\det({\rm 
 Id}-P(\gamma_{\tau'}))|e^{\lambda \ell(\gamma_{\tau'})}} 
 {\rm Tr}(\rho([\gamma_{\tau'}]))d\tau'\right).\] 
\end{corr}
Note that the reason for the value at $0$ of the twisted Ruelle zeta function being locally constant appears clearly in this last statement. Yet, this formula is only valid for a large real part of $\lambda$ and the rest of the proof consists in showing that this formula meromorphically extends at $\lambda=0$. This is where our hypothesis on Ruelle resonances at $0$ will play a role.

In order to interpret this variation formula spectrally, we already observe that we recognize here the Jacobian terms appearing in
the Guillemin trace formula~\cite[p.~315]{GS77} when we consider the action on the full vector bundle of differential forms of degree $k$ -- see~\eqref{guillemintrace} below.

\begin{rem} As was already said, the reason for the introduction of the bundle map $A_\tau$ was due to the lack of smoothness of $q_\tau$ in the non-contact case. In the contact case, the formula could be written more simply as follows:
$$\frac{\zeta_{\tau,\rho}(\lambda)}{\zeta_{0,\rho}(\lambda)}=\exp\left(\lambda\int_0^{\tau}\sum_{k=0}^n(-1)^{k}k\sum_{\gamma_{\tau'}}\frac{\ell^\sharp(\gamma_{\tau'})}{\ell(\gamma_{\tau'})}
 \frac{\left(\int_{\gamma_{\tau'}}q_{\tau'}\right){\rm Tr}\left(\wedge^k 
 d\varphi^{\tau'}_{\ell(\gamma_{\tau'})}\right)}{|\det({\rm 
 Id}-P(\gamma_{\tau'}))|e^{\lambda \ell(\gamma_{\tau'})}} 
 {\rm Tr}(\rho([\gamma_{\tau'}]))d\tau'\right),$$
 where we used Lemma \ref{l:regularity-physical-region} and the formula (see ~\cite[p.~50]{F87})
 $$\varepsilon_{\gamma_\tau}=\frac{\det({\rm Id}-P(\gamma_\tau))}{|\det({\rm 
 Id}-P(\gamma_\tau))|}=\sum_{k=0}^n(-1)^k\frac{{\rm Tr}\left(\wedge^k P(\gamma_\tau)\right)}{|\det({\rm Id}-P(\gamma_\tau))|}=\sum_{k=0}^{n}(-1)^{k+1}k\frac{{\rm Tr}\left(\wedge^k d\varphi^\tau_{\ell(\gamma_\tau)}\right)}{|\det({\rm Id}-P(\gamma_\tau))|}.$$
Under that form, this variation formula  would be amenable to our microlocal techniques as $q_\tau$ is smooth and the formulas below would be slightly simpler. Yet, in the general case, we have to rewrite $q_\tau$ using $A_\tau$ as we did in Lemma~\ref{l:shu-trick}. 
\end{rem}

\section{Variation formula in the non-convergent region}\label{s:traceflat}

We recall that \cite{GLP, DyZw13} show that $\zeta_{\tau,\rho}(\lambda)$ admits a meromorphic continuation $\la\in \cc$. This was achieved by relating the Ruelle zeta function to some \emph{flat trace} of some operator. We will use similar ideas to rewrite $\frac{\zeta_{\tau,\rho}(\lambda)}{\zeta_{0,\rho}(\lambda)}$ in terms of flat traces by analysing 
\begin{equation}\label{Fktau}
F_\tau^{(k)}(\la):=\sum_{\gamma_{\tau}}\frac{\ell^\sharp(\gamma_{\tau})}{\ell(\gamma_{\tau})}
 \frac{\left(\int_{\gamma_{\tau}}{\rm Tr}\left(A_{\tau}^{(k)}\left(\wedge^k 
 d\varphi^{\tau}_{\ell(\gamma_{\tau})}\right)\right)\right)e^{-\lambda 
 \ell(\gamma_{\tau})}}{|\det({\rm Id}-P(\gamma_{\tau}))|}
 {\rm Tr}(\rho([\gamma_{\tau}])).
\end{equation}
Note that, in these references, the meromorphic extension was proved under some orientability hypothesis but this assumption can be removed by introducing the orientation 
index in the definition of the Ruelle zeta function as we did.

\subsection{Reformulation via distributional traces} 

Let us start with a brief reminder on \emph{flat traces}. First, if $M$ is a compact manifold and $\Gamma\subset T_0^*M$ a closed conic subset, we define, following H\"ormander \cite[Section 8.2]{Ho}, the space   
\[\mc{D}'_\Gamma(M):=\{ u\in \mc{D}'(M); {\rm WF}(u)\subset \Gamma\}.\]
Its topology is described using sequences in \cite[Def. 8.2.2.]{Ho}, we will recall it later.
Denote by $\Delta$ the diagonal in $\mc{M}\x\mc{M}$ and by 
$$N^*\Delta:=\left\{(x,x,\xi,-\xi):\xi\neq 0\right\}\subset T_0^*(\mc{M}\x\mc{M}).$$
We fix a smooth density on $\mathcal{M}$, so that distributions can be viewed with scalar values rather than densities. If $E\to \mc{M}$ is a vector bundle over $\mc{M}$, the Atiyah--Bott 
flat trace of a $K\in  \mc{D}'_{\Gamma}(\mc{M}\x\mc{M}; E\otimes E^*)$ 
with $\Gamma\cap N^*\Delta=\emptyset$ is defined by
\[ \Tr^\flat(K):=  \cjg \Tr(i_{\Delta}^*K),1\cjd \]
where $i_\Delta: \mc{M}\to \mc{M}\x\mc{M}$ is the natural inclusion map $i_{\Delta}(x):=(x,x)$ and $\Tr$ denotes the 
local trace of endomorphisms  ${\rm End}(E)=E\otimes E^*$, so that $\Tr(i_{\Delta}^*K)\in \mc{D}'(\mc{M})$.
\begin{lemm}\label{continuityTrb}
For each closed conic subset $\Gamma\subset T^*(\mc{M}\x\mc{M})$ satisfying $\Gamma\cap N^*\Delta=\emptyset$, the flat trace $\Tr^\flat$ is a sequentially continuous linear form 
\[  \Tr^\flat : \mc{D}'_{\Gamma}(\mc{M}\x\mc{M}; E\otimes E^*)\to \cc \]
with respect to the topology of $\mc{D}'_{\Gamma}(\mc{M}\x\mc{M};E\otimes E^*)$. 
\end{lemm}
\begin{proof} 
This follows directly from continuity of the pullback from $\mathcal{D}^\prime_\Gamma(\mc{M}\times \mc{M};E\otimes E^*)\mapsto \mathcal{D}^{\prime}(\mc{M})$
\cite[Theorem 8.2.4]{Ho} and continuity of the pairing against $1$.
\end{proof} 
For an operator 
$B: C^\infty(\mc{M};E)\to \mc{D}'(\mc{M};E)$ with Schwartz kernel $K_B$ satisfying $K_B\in 
\mc{D}'_{\Gamma}(\mc{M}\x \mc{M}; E\otimes E^*)$ for some $\Gamma$ with $\Gamma\cap N^*\Delta=\emptyset$, 
we write 
\[\Tr^\flat(B):= \Tr^\flat(K_B). \]
Then, by a slight extension of the Guillemin trace formula~\cite[p.~315]{GS77}, we have 
\begin{equation}\label{guillemintrace}
\text{Tr}^\flat\left(A^{(k)}_{\tau}e^{-t\mathbf{X}_{\tau}}|_{\Omega^k(M,E)} \right)= \sum_{\gamma_\tau}\frac{\ell^\sharp(\gamma_{\tau})}{\ell(\gamma_{\tau})}
 \frac{\int_{\gamma_{\tau}}{\rm Tr}\left(A_{\tau}^{(k)}\left(\wedge^k 
 d\varphi^{\tau}_{\ell(\gamma_{\tau})}\right)\right)}{|\det(\text{Id}-P(\gamma_{\tau}))|}
 \operatorname{Tr}(\rho([\gamma_{\tau}])) \delta(t-\ell(\gamma_{\tau})), 
\end{equation}
in $\mathcal{D}^\prime(\mathbb{R}_{>0})$, where this equality holds for every $\tau$ such that $X_{\tau}\in\ml{U}(X_0)$ and where the sum runs over all closed orbits. 
 We choose $t_0>0$ so that there is some $c>0$ uniform in $\tau$ ($\tau$ is also close enough to $0$) such that $\min_{x\in \mc{M}} d_g(x,\varphi^\tau_{t_0}(x))\geq c$
and define the meromorphic family of operators (well-defined by Proposition \ref{p:ruelle-spectrum}) 
\begin{equation}\label{e:shifted-resolvent}
Q_{\tau}(\la):=
e^{-t_0{\bf X}_\tau}(-{\bf X}_\tau-\la)^{-1}.
\end{equation} 
Our assumption on $t_0$ will be used later on when we will bound the wavefront set of the kernel of the operator. By the same arguments as in~\cite[\S~4]{DyZw13}, we obtain that 
${\rm Tr}^{\flat}(A_{\tau}^{(k)}Q_{\tau}(\la)|_{\mc{E}^k})$ is well-defined for each small $\tau$ as a meromorphic function in $\la\in\cc$ and 
\begin{equation}\label{e:magic-formula}
\textrm{ if } {\rm Re}(\la)>C_0, \quad F_\tau^{(k)}(\la)={-}
e^{-\la t_0}{\rm Tr}^{\flat}\left(A_{\tau}^{(k)}Q_{\tau}(\la)|_{\ml{E}^k}\right)
\end{equation}
with $C_0>0$ given by Lemma~\ref{l:regularity-physical-region}.
\begin{rem} Again, in the contact case, we would have the simpler formula
 $$\textrm{ if } {\rm Re}(\la)>C_0, \quad F_\tau^{(k)}(\la)=k
 e^{-\la t_0}{\rm Tr}^{\flat}\left(q_{\tau}Q_{\tau}(\la)|_{\ml{E}^k}\right).$$
\end{rem}

\subsection{Proof of Theorem~\ref{t:maintheo}}\label{ss:proof-variation-formula}
The proof of Theorem~\ref{t:maintheo} will follow directly from Corollary \ref{corrolairevar} and the following 
\begin{theo}\label{continofFk}
Assume that $X_0\in\mc{A}$ is such that ${\bf X}_0$ has no Pollicott-Ruelle resonance at $\la=0$ and let $\mc{Z}\subset \cc$ be an open connected subset 
containing $0$ and a point inside the region $\{{\rm Re}(\lambda)> C_{X_0}\}$ and such that ${\bf X}_0$ has no Pollicott-Ruelle resonance in $\overline{\ml{Z}}$. 
Then, there exists a neighborhood $\mc{U}(X_0)\subset \mc{A}$ of $X_0$ such that
\\ 
1) the operator 
$(-{\bf X}-\la)^{-1}$ of Proposition \ref{p:ruelle-spectrum}  is holomorphic in $\mc{Z}$ for all $X\in \mc{U}(X_0)$.\\
2) if $\tau \mapsto X_\tau\in \mc{U}(X_0)$ is a smooth map with $X_\tau|_{\tau=0}=X_0$, then 
$\tau\mapsto 
{\rm Tr}^{\flat}(A_{\tau}^{(k)}Q_{\tau}(\la)|_{\mc{E}^k})$ is continuous with values in 
${\rm Hol}(\mc{Z})$, with $A_{\tau}^{(k)}$ defined by \eqref{Atauk}. 
\end{theo}
Take $B_k(X_0,\eps):=\{X\in \mc{A}; \|X-X_0\|_{C^k}\leq \eps\}$ contained in the neighborhood $\mc{U}(X_0)$ of Theorem \ref{continofFk}, for some $k\in\nn,\eps>0$, and for $X\in B_{k}(X_0,\eps)$ define $X_\tau:=X_0+\tau(X-X_0)$ for $\tau\in(-\delta,1+\delta)$ with $\delta>0$ small so that $X_\tau\in B_{k}(X_0,\eps)$. Now each ${\bf X}_\tau$ has no resonances in $\mc{Z}$ and 2) in Theorem \ref{continofFk} with \eqref{e:magic-formula} show that $\tau\mapsto F_\tau^{(k)}(\la)$ can be extended as a continuous family of functions  in ${\rm Hol}(\mc{Z})$ for $\tau\in [0,1]$. Corollary \ref{corrolairevar} then shows that $\zeta_{\tau,\rho}(\la)/\zeta_{0,\rho}(\la)$ admits a holomorphic extension in $\mc{Z}$ with $\zeta_{\tau,\rho}(0)=\zeta_{0,\rho}(0)$. Thus $\zeta_{X,\rho}(0)=\zeta_{X_0,\rho}(0)$. The proof of Theorem \ref{continofFk} will be given in the next section.

\section{Continuity of the resolvent and Proof of Theorem 4}\label{s:microlocal}

The purpose of this section is to prove the properties of the Schwartz kernel of the resolvent 
that were used in the proof of Theorem~\ref{t:maintheo}. We are interested in the continuity with respect to $\tau$ of the flat trace of the operator
\begin{equation}\label{Ltau}
Q_{\tau}(\la):=
e^{-t_0{\bf X}_\tau}(-{\bf X}_\tau-\la)^{-1}
\end{equation} 
where we recall that we chose $t_0>0$ so that there is some $C>0$ uniform in $\tau$ (here $\tau$ is close enough to $0$) such that
\[ \min_{x\in \mc{M}} d_g(x,\varphi^\tau_{t_0}(x))\geq C\]
where $d_g$ is the Riemannian distance induced by a metric $g$. The arguments used here are variations on the microlocal proofs of Faure-Sj\"ostrand 
in~\cite{FaSj} and Dyatlov-Zworski in~\cite{DyZw13}. The continuity of the 
resolvent also follows from Butterley-Liverani \cite{BL07}.
For $k\in\rr$, we will write $\Psi_h^{k}(\mc{M};\mc{E})$ for the space of semi-classical pseudo-differential operators  \cite[Chapter 14.2]{Zw} (on sections of $\mc{E}$) with symbols in the class 
$S^k_h(T^*\mc{M};\mc{E})$ defined by: 
$a_h\in S^k_h(T^*\mc{M};\mc{E})$ if $a_h\in C^\infty(T^*\mc{M};{\rm End}(\mc{E}))$ satisfies 
$|\pl_x^\alpha\pl_\xi^\beta a_h(x,\xi)|\leq C_{\alpha\beta}\cjg \xi\cjd^{k-|\beta|}$ with $C_{\alpha\beta}$ independent of $h$. 
As mentionned before, we also take a semi-classical quantisation ${\rm Op}_h$ mapping $S^k_h(T^*\mc{M};\mc{E})$ to $\Psi^k_h(\mc{M};\mc{E})$.
The operators in the class $\Psi^{k}(\mc{M};\mc{E}):=\Psi^{k}_{h_0}(\mc{M};\mc{E})$ for some fixed small $h_0>0$ are called pseudo-differential operators.
We introduce the family of $h$-pseudodifferential operators:
\begin{equation}\label{defofPX}
P_{X}(h,\lambda):=\mathbf{A}_h(N_0,X)(-h\mathbf{X}-h\lambda)\mathbf{A}_h(N_0,X)^{-1},
\end{equation}
where $\mathbf{A}_h(N_0,X):=\mathbf{A}_h(N_0,20N_0,X)$ was defined in Equation~\eqref{e:sobolev-weight} and Remark~\ref{r:N0N1}. All along this section, $N_0$ will be chosen large enough (say at least equal to $1$).

\subsection{Continuity of the resolvent for families of Anosov flows}

For the first part of Theorem \ref{continofFk} we prove:
\begin{prop}\label{p:contofresolvent} 
Let $X_0$ and $\mc{Z}$ chosen as in Theorem \ref{continofFk}.
There exist a neighborhood $\ml{U}(X_0)$ of $X_0$, $h_0>0$ and $C>0$ 
such that, for every $0<h<h_0$, 
and for every $X\in\ml{U}(X_0)$, the map $\lambda\in\ml{Z}\mapsto P_{X}(h,\lambda)^{-1}\in\ml{L}(L^2,L^2)$ 
is holomorphic and
\begin{equation}\label{boundonPXhla}
\forall\lambda\in\ml{Z},\quad \left\|P_{X}(h,\lambda)^{-1}\right\|_{H^1_h\rightarrow L^2}\leq Ch^{-1-100N_0}.
\end{equation}
Moreover, for every $0<h<h_0$, the following map is continuous
\[
X\in\ml{U}(X_0)\mapsto P_{X}(h,\lambda)^{-1}\in\operatorname{Hol}\left(\ml{L}(H^{1}_h,L^2)\right).
\]
\end{prop}
\begin{proof} 
 In order to prove this Proposition, we need to review the proofs from~\cite[p.340-345]{FaSj} -- see also paragraph~5 from this reference or~\cite{DyZw13} for a semiclassical formulation 
 as described here. 
 Note already from Proposition~\ref{p:ruelle-spectrum} that, for every $X\in \mc{U}(X_0)$, 
 $\lambda\in\ml{Z}\mapsto P_{X}(h,\lambda)^{-1}\in\ml{L}(L^2,L^2)$ is meromorphic. 
 
 Recall from~\cite[Lemma~5.3]{FaSj} that
 \begin{equation}\label{decofPXh}
 P_{X}(h,\lambda)=\Op_h\left(\left(-iH_{X}-h\lambda+h\left\{H_{X},G_{X}^{N_0}\right\}\right)\text{Id}\right)+\ml{O}_{X}(h)+\ml{O}_{m_{X}^{N_0}}(h^2),
 \end{equation}
 where $H_X(x,\xi)=\xi(X (x))$ and where the remainders are understood as 
 bounded operator on $L^2(\ml{M};\ml{E})$. Only the second 
 remainder depends on the choice of the order function, and both remainders can be made uniform in terms of $X\in\ml{U}(X_0)$ thanks to Lemma~\ref{l:escape-function}. 
 Following~\cite[\S~3.3]{FaSj}, one can introduce an operator $\hat{\chi}_0={\rm Op}_h(\chi_0{\rm Id})$ 
 in $\Psi_h^0(\ml{M};\ml{E})$ depending only on $X_0$ with $\chi_0\geq 0$ and so that ($c_0$ is the constant from Lemma~\ref{l:escape-function})
\begin{equation}\label{XGbig}
\forall(x,\xi)\in T^*\ml{M} ,\quad\left\{H_{X},G_{X}^{N_0}\right\}-\chi_0(x,\xi)^2\leq -2c_0N_0\end{equation}
\begin{rem}\label{r:flexibility-absorbing-potential}
Note that we have some flexibility in the choice of the operator $\widehat{\chi}_0$. Besides the fact that it belongs to $\Psi_h^0(\ml{M},\ml{E})$, the only requirements we shall need are
\begin{itemize}
\item $\chi_0^2=C_{N_0}+2c_0N_0$ (inside a 
small conic neighborhood of $E_0^*(X_0)$), where $C_{N_0}>0$ is the 
 uniform constant from \eqref{e:bound-derivative-escape-function}, 
 \item outside a slightly larger conic neighborhood of $E_0^*(X_0)$, 
 ${\rm supp}(\chi_0)$ is contained in $\{\|\xi\|\leq 3R/2\}$ where $R$ is the parameter from Lemma~\ref{l:escape-function},
\item $\chi_0$ satisfies \eqref{XGbig} in $\{\|\xi\|\leq R\}$.
\end{itemize}
\end{rem}
Next we let $\hat{\chi}_1={\rm Op}_h(\chi_1\text{Id})\in \Psi_h^{0}(\mc{M})$ with $\chi_1\in C_0^\infty(T^*\mc{M},\rr_+)$ where $\chi_1$ is a function of $|\xi|$ satisfying $\supp(\chi_1)\subset \{\|\xi\|\leq 3R/2\}$, 
and $\chi_1(x,\xi)=1$ for $\|\xi\|\leq R$, and we define\footnote{The operator $\hat{\chi}_1^*\hat{\chi}_1$ is not necessary for this proof but will be useful for the wavefront set analysis later.} 
\begin{equation} \label{defofchi}
\hat{\chi}:= \hat{\chi}_1^*\hat{\chi}_1+h\hat{\chi}_0^*\hat{\chi}_0 \in \Psi^0_h(\mc{M};\mc{E}).
\end{equation}
One can apply the semiclassical G\"arding inequality to verify that there exists a constant $C_0>0$ (uniform in $X\in \mc{U}(X_0)$) such that, for every $u\in\mathcal{C}^\infty(\mathcal{M},\mathcal{E})$, $\text{Re}(\lambda)>C_0-c_0N_0$ and every $0<h<1$,
$$\text{Re}\langle (P_X(h,\lambda)-h\hat{\chi}_0^*\hat{\chi}_0)u,u\rangle_{L^2}\leq -c_0N_0h\|u\|_{L^2}^2.$$
This shows that $$\left(P_X(h,\lambda)-\hat{\chi}\right)^{-1}:L^2(\ml{M},\ml{E})\rightarrow L^2(\ml{M},\ml{E})$$
is bounded for $\text{Re}(\lambda)>C_0-c_0N_0$. Moreover, we get a uniform upper bound: there is $C>0$ such that for $\la$ as above
\begin{equation}\label{e:resolventbound}
\forall X\in \mc{U}(X_0), \, \forall 0<h<1,\quad\left\|\left(P_X(h,\lambda)-\hat{\chi}\right)^{-1}\right\|_{L^2\rightarrow L^2}\leq Ch^{-1}.
\end{equation}
By  adding a constant  
$s\in [-1,1]$ to the order function $m_X^{N_0}$, the same argument as above works and we can pick the operators $\hat{\chi}_0$ and $\hat{\chi}_1$ independently of $s\in[-1,1]$. 
Since the consideration of $P_{\tau}(h,\lambda)-\hat{\chi}$ acting on $H_h^s(\mc{M};\ml{E})$ 
is equivalent to its conjugation by ${\rm Op}_h((1+ f)^{s})$, it implies that
\begin{equation}\label{e:bound-Sobolev}h(P_X(h,\lambda)-\hat{\chi})^{-1}: H_h^s(\mc{M};\ml{E})\to H_h^s(\mc{M};\ml{E})\end{equation}
is uniformly bounded in $(\la,X,h)$ for all $(X,\la)$ as before and all $h>0$ small. In order to study the continuity, we first write
\begin{eqnarray*}
(P_X(h,\lambda)-\hat{\chi})^{-1}& = &(P_{X_0}(h,\lambda_0)-\hat{\chi})^{-1}\\
&+&(P_X(h,\lambda)-\hat{\chi})^{-1}\left(P_{X_0}(h,\lambda_0)-P_{X}(h,\lambda)\right)(P_{X_0}(h,\lambda_0)-\hat{\chi})^{-1}. 
\end{eqnarray*}
Thanks to the Calder\'on-Vaillancourt Theorem~\cite[Th.~5.1]{Zw}, one knows that 
\[\left\|P_{X_0}(h,\lambda_0)-P_{X}(h,\lambda)\right\|_{ H^{1}_h\rightarrow L^2}\leq C\|X-X_0\|_{C^k}+h|\lambda-\lambda_0|\]
for some $k\geq 1$ large enough (depending only on the dimension of $\ml{E}$) and for some $C>0$ independent of $h$, $X$ and $\lambda$. Hence, combined with~\eqref{e:bound-Sobolev}, we find that 
the map $(X,\lambda)\mapsto (P_{X}(h,\lambda)-\hat{\chi})^{-1}\in\ml{L}( H^{1}_h,L^2)$ is continuous.

Next, as in~\cite[p.~344]{FaSj}, one can construct 
$E_X(h,\la)\in \Psi_h^{-1}(\mc{M};\ml{E})$ whose principal symbol is supported in a conic neighborhood of $E_0^*(X_0)$ 
so that 
\[(P_X(h,\la)-\hat{\chi})E_X(h,\la)={\rm Id}+S_X(h,\la),\quad  
E_X(h,\la)(P_X(h,\la)-\hat{\chi})={\rm Id}+T_X(h,\la)\] 
with $S_X(h,\la)$ and 
$T_X(h,\la)$ both in $\Psi_h^0(\mc{M};\ml{E})$ 
such that the support of their principal symbols intersects $\supp(\chi_0)\cup\supp(\chi_1)$ 
inside a compact region of $T^*\ml{M}$ which is independent of $(X,\lambda)$. Note that all these pseudodifferential operators depend continuously in $(X,\la)$ (these are just parametrices in the elliptic region). Then,
\begin{equation}\label{Kato}
K_X(h,\la):=\hat{\chi}(P_X(h,\la)-\hat{\chi})^{-1}=\hat{\chi}E_X(h,\la)-\hat{\chi}T_X(h,\la) (P_X(h,\la)-\hat{\chi})^{-1}
\end{equation}
is compact as $\hat{\chi}E_X(h,\la)\in \Psi_h^{-1}(\mc{M};\ml{E})$ and 
$\hat{\chi}T_X(h,\la)\in \Psi_h^{-1}(\mc{M},\ml{E})$. 

This operator (viewed as an element of $\ml{L}(H_h^1,H_h^1)$) 
depends continuously on $(X,\lambda)$. Moreover, from our upper bound on the modulus of continuity of $(X,\lambda)\mapsto (P_{X}(h,\lambda)-\hat{\chi})^{-1}$, we get
$$\|K_X(h,\la)-K_{X_0}(h,\la_0)\|_{H_h^1\rightarrow H_h^1}\leq \frac{1}{h^2}\omega(|\la-\la_0|,||X-X_0||_{C^k}),$$
where $\omega(x,y)$ is independent of $(h,X,\lambda)$ and verifies $\omega(x,y)\rightarrow 0$ as $(x,y)\rightarrow 0$. With this family of compact operators, we get the 
identity (as meromorphic operators in $\la$ on $H_h^1$)
\begin{equation}\label{relationentreres} 
\text{Id}+K_X(h,\la)=P_X(h,\la)(P_X(h,\la)-\hat{\chi})^{-1}.
\end{equation} 
Now, from the definition of $\ml{Z}$, we know that, for every $\lambda\in \ml{Z}$, $(\text{Id}+K_{X_0}(h,\la))$ is invertible in $\ml{L}(H_h^1,H_h^1)$. Thus, by continuity of the inverse map, 
we can then conclude that this remains true for any $||X-X_0||_{C^k}$ small enough  uniformly for $\lambda\in \ml{Z}$ (as $P_{X_0}(h,\lambda)$ remains invertible for $\lambda$ 
in $\overline{\ml{Z}}$).  The neighborhood depends a priori on $h$ but it can be made uniform in $h$, as all the operators $P_X(h,\lambda)$ are conjugated for different values of $h$,
$$\forall X\in\mc{U}(X_0),\ P_{X}(h,\lambda)=\frac{h}{h_0}\mathbf{A}_h(N_0,X)\mathbf{A}_{h_0}(N_0,X)^{-1}P_{X}(h_0,\lambda)\mathbf{A}_{h_0}(N_0,X)\mathbf{A}_h(N_0,X)^{-1}.$$
 It now only remains to verify the upper bound on the norm of the resolvent. For that purpose, we can fix $h=h_0>0$ with $h_0$ small enough. The above proof shows that 
$P_{X}(h_0,\lambda)$ is uniformly bounded (for $X\in\ml{U}(X_0)$ and $\lambda\in\overline{\ml{Z}}$) as an operator from $H^1_{h_0}$ to $L^2$. 
We observe that  for $h,h_0$ fixed, the operators $\mathbf{A}_{h_0}(N_0,X)^{-1}$ and $\mathbf{A}_{h}(N_0,X)$ belong to the class of (non semiclassical) pseudodifferential operators with variable order (see for example \cite[App.~A]{FRS}). Their order are respectively $-m_X^{N_0}(x,h_0\xi)$ and $m_X^{N_0}(x,h\xi)$ while their principal symbols are given by $(1+f)^{-m_X^{N_0}}(x,h_0\xi)$ and  
$(1+f)^{m_X^{N_0}}(x,h\xi)$. 
By the composition rule of pseudodifferential operators, 
their product has order $0$  with principal symbol 
\[ \frac{(1+f(x,h\xi))^{m_X^{N_0}(x,h\xi)}}{(1+f(x,h_0\xi))^{m_X^{N_0}(x,h_0\xi)}},\]
and its full local symbol in charts is given by derivatives of these symbols.
Now, applying the Calder\'on-Vaillancourt Theorem, we know that the $\mc{L}(L^2)$ norm of $\mathbf{A}_{h}(N_0,X)\mathbf{A}_{h_0}(N_0,X)^{-1}$ is bounded by a finite number of derivatives of the full symbol (written in charts). Due to the fact that these are symbols in $\xi$ and using that $-2N_0\leq m_X\leq 40 N_0$ and that $f$ is homogeneous of degree $1$ at infinity, 
it is direct to check that (as $N_0\geq 1$)
$$\left\|\mathbf{A}_h(N_0,X)\mathbf{A}_{h_0}(N_0,X)^{-1}\right\|_{L^2\rightarrow L^2}+\left\|\mathbf{A}_{h_0}(N_0,X)\mathbf{A}_{h}(N_0,X)^{-1}\right\|_{H_h^1\rightarrow H_{h_0}^1}
\lesssim h^{-50N_0},$$
from which we can deduce the announced upper bound on the norm of the resolvent.
\end{proof}

\subsection{Wavefront set of the Schwartz kernel of the resolvent}

The next part consists in bounding locally uniformly in $(\tau,\la)$ the Schwartz kernel of the operator $Q_{\tau}(\la)$ defined in \eqref{Ltau}.

First, let us introduce a bit of terminology. Let $M$ be a compact manifold (in practice, we 
take $M=\mc{M}$ or $M=\mc{M}\x\mc{M}$). We refer for example to \cite[Appendix C.1]{DyZw13} for a summary of the notion of wavefront set ${\rm WF}(A)\subset T_0^*M$ (resp. ${\rm WF}(u)\subset T_0^*M$) of an operator $A\in \Psi^{k}(M)$ (resp. of a distribution $u\in \mc{D}'(M)$). 
For $\Gamma\subset T^*_0M$ a closed conic set, we say that a family $u_\tau\in \mc{D}'(M)$ with $\tau\in [\tau_1,\tau_2]\subset \rr$ 
is bounded in $\mathcal{D}^\prime_\Gamma$ if it is bounded in $\mathcal{D}^\prime$ and 
for each $\tau$-independent $A\in \Psi^0(M)$ with ${\rm WF}(A)\cap \Gamma=\emptyset$, 
\[ \forall N\in\nn, \exists C_{N,A}>0, \forall \tau\in [\tau_1,\tau_2], \quad ||A(u_\tau)||_{H^N}\leq C_{N,A}.\]
This can also be described in terms of Fourier transform in charts (see \cite[Appendix C.1]{DyZw13}).
Similarly, we refer to \cite[Appendix C.2]{DyZw13} for a summary on the semi-classical wavefront set 
${\rm WF}_h(A)\subset \bbar{T^*M}$ (resp. ${\rm WF}_h(u)\subset \bbar{T^*M}$) of an operator  
$A={\rm Op}_h(a_h)\in \Psi_h^{k}(M)$ (resp. of a $h$-tempered family of distributions $u_h\in \mc{D}'(M)$); here $\bbar{T^*M}$ 
denotes the fiber-radially compactified cotangent bundle (see \cite[Section 2.1]{Va}). 

We recall from \cite[Definition 8.2.2]{Ho} the topology of $\mc{D}'_\Gamma(M)$: 
a sequence $u_\tau\in \mc{D}'_\Gamma(M)$ converges to $u_{\tau_0}$ in $\mc{D}'_\Gamma(M)$ as $\tau\to \tau_0$ if $u_\tau\to u_{\tau_0}$ in $\mc{D}'(M)$ and $(u_\tau)_\tau$ is bounded in $\mathcal{D}^\prime_\Gamma$. 

We note that all these properties hold the same way for sections of vector bundles.

Next, we recall a result which is essentially Lemma 2.3 in \cite{DyZw13} characterising the wave-front set of a family  $K_\tau \in \mc{D}'(\mc{M}\x\mc{M};\ml{E}\otimes \ml{E}^\prime)$, but uniformly in the parameter $\tau$. 
We shall use a semi-classical parameter $h>0$ for this characterisation.
\begin{lemm}\label{Dyzwlemma-Lagrangian}
Let $K_\tau \in \mc{D}'(\mc{M}\x\mc{M};\mc{E}\otimes \mc{E}^{\prime})$ be an $h$-independent bounded family depending on $\tau\in [\tau_1,\tau_2]$ and let $\mc{K}_\tau$ be 
the associated operator on $\mc{M}$. Let $\Gamma\subset 
T_0^*(\mc{M}) \times T_0^*(\mc{M})$ be a fixed closed conic set, independent of $\tau$. 

Assume that for each point $(y,\eta, z,-\zeta)\in (T_0^*\mc{M}\x T_0^*\mc{M})\setminus \Gamma$, 
there are small relatively compact neighborhoods $U$ of $(z,\zeta)$ and $V$ of $(y,\eta)$ in $T^*\mc{M}$ such that, for every $N\geq 1$ and for every $\tau$-independent $B_h\in \Psi_h^0(\mc{M},\ml{E})$ microlocally supported inside $ V$, there exist $C_{N,B}>0$ and $k_{N,B}>0$ with the following property:\\
For every function $a\in\mathcal{C}^{\infty}(\mathcal{M},\mathcal{E})$ compactly supported near $z$ and for every $S\in\mathcal{C}^{\infty}(\mathcal{M},\mathbb{R})$ such that\footnote{This implies that the Lagrangian states $(f_h)_{0<h\leq 1}$ verifies ${\rm WF}_h(f_h)\subset U$~\cite[p.~190]{Zw}.}
$$\{(x,d_xS):x\in\operatorname{supp}(a)\}\subset U,$$
the Lagrangian states $f_h:=ae^{\frac{iS}{h}}$ verify 
\begin{equation}\label{uniformKtau} 
\forall \tau, \forall h \in(0,1) \quad  ||B_h\mc{K}_\tau f_h||_{L^2}\leq C_{N,B}\|(a,S)\|_{\mathcal{C}^{k_{N,B}}}h^{N}.
\end{equation}
Then, $(K_\tau)_\tau$ is a bounded family of distributions in 
$\mc{D}'_\Gamma(\mc{M}\x\mc{M};\ml{E}\otimes \ml{E}^{\prime})$.
\end{lemm}
\begin{proof} 
The proof is readily the same as the first part of the proof of~\cite[Lemma 2.3, App.~C.2]{DyZw13} by just adding the $\tau$ dependence. Compared with that reference, note that, as the kernel $K_\tau$ is $h$-independent, it is sufficient to consider points inside $S^*\mathcal{M}\times S^*\mathcal{M}$ (the wavefront set being a conical subset when it is $h$-independent).  
\end{proof}

\subsubsection{Main technical result}

We shall now prove that the kernel of the resolvent is uniformly bounded in $\mc{D}'_{\Gamma}(\mc{M}\x\mc{M};\mc{E}\otimes \mc{E}^\prime)$, 
where $\Gamma$ is a closed cone that does not intersect $N^*\Delta$.

\begin{prop}\label{p:boundonWF} 
There exist a small neighborhood $\ml{U}(X_0)$ of $X_0$ in the $C^{\infty}$-topology and a closed conic set $\Gamma\subset T_0^*(\mc{M}\x\mc{M})$ not intersecting $N^*\Delta$ such that, 
for every $\tau\mapsto X_\tau$ as in 2) of Theorem \ref{continofFk},
$$(\tau,\lambda)\in[-\delta,\delta]\times\overline{\ml{Z}}
\mapsto Q_{\tau}(\lambda)(.,.)\in\ml{D}^\prime_{\Gamma}(\ml{M}\times\ml{M}, \mc{E}\otimes \mc{E}')$$
is bounded, where $\delta>0$ is small enough to ensure
that $X_{\tau}\in\ml{U}(X_0)$ for all $\tau\in[-\delta,\delta]$.
\end{prop}

\subsubsection{Proof of~Proposition \ref{p:boundonWF}}

Thanks to Proposition~\ref{p:contofresolvent}, we already know that the Schwartz kernel of $Q_\tau(\la)$  is uniformly bounded 
on $\ml{D}^{\prime}(\ml{M}\times\ml{M}; \mc{E}\otimes \mc{E}')$ and 
$Q_\tau(\la)\to Q_{\tau_0}(\la_0)$ in this space as $(\tau,\la)\to (\tau_0,\la_0)$ for $|\tau_0|\leq \delta, \la_0\in \bbar{\mc{Z}}$. 
Hence, it only remains to show that the family is bounded in $\ml{D}^{\prime}_{\Gamma}(\ml{M}\times\ml{M}, \mc{E}\otimes \mc{E}')$. 
We shall use the criteria of Lemma~\ref{Dyzwlemma-Lagrangian} to get a bound on the kernel of the resolvent and, up to some details of presentation, 
we will follow partly~\cite{DyZw13} by combining with \cite{FaSj} and we shall verify that everything is bounded uniformly in the parameter $\tau$.

Recall that, up to multiplication by $h$, our kernel is $h$-independent. Hence, it is sufficient to test one covector in every direction of $T_0^*\mathcal{M}$ and we take some $R>0$ larger than the $R$ 
appearing in Lemma~\ref{l:escape-function} and we fix some point $(z,\zeta)$ in $T^*\ml{M}$ such that $2R\leq \|\zeta\|\leq 4R.$ 
Let $U$ be a small enough neighborhood of $(z,\zeta)$ in $T^*\mc{M}$ so that $U_{t_0,\delta}:=\bigcup_{|\tau|\leq \delta}\Phi^\tau_{t_0}(U)$ satisfies 
$ \bbar{U} \cap \bbar{U_{t_0,\delta}}=\emptyset$
where the existence of $U$ is guaranteed by the choice of $t_0$.  
We also fix $t_0$ small enough so that $\bbar{U_{t_0,\delta}}\cap \{\|\xi\|\leq 3R/2\}=\emptyset$.
Let $f_h=ae^{\frac{iS}{h}}\in C^\infty(\mc{M};\mc{E})$ with $a$ a smooth function compactly supported in a small neighborhood of $z$ and $S$ a smooth (real valued) function such that
$\{(x,d_xS):x\in\text{supp}(a)\}\subset U$. Define 
\[\tilde{f}_h(\tau):=he^{-t_0\mathbf{X}_{\tau}}f_h\]
which verifies that ${\rm WF}_h(\tilde{f}_h(\tau))\subset \bbar{U_{t_0,\delta}}$~\cite[Th.~8.14]{Zw} uniformly in $\tau$ (with the involved constants depending on a finite number derivatives of $a$ and $S$), thus  not intersecting $\overline{U}$. Let
$$u_h(\tau,\lambda)=(-h\mathbf{X}_{\tau}-h\lambda)^{-1}\tilde{f}_h(\tau),$$
 where $|\tau|\leq \delta$ for some small $\delta>0$ and where $\lambda$ varies in $\ml{Z}$. 
 
 We now conjugate the operators with 
 $\mathbf{A}_h(N_0,\tau)$ in order to work with the more convenient operator $P_{\tau}(h,\lambda)$ defined in \eqref{defofPX} (with $X=X_\tau)$, i.e.
\[\begin{gathered}
P_{\tau}(h,\lambda)\tilde{u}_h(\tau,\la)=\tilde{F}_h(\tau), \,\,\, \textrm{ with } \\
\quad \tilde{u}_h(\tau,\lambda):=\mathbf{A}_h(N_0,\tau)u_h(\tau,\lambda), \quad 
\tilde{F}_h(\tau):=\mathbf{A}_h(N_0,\tau)\tilde{f}_h(\tau).
\end{gathered}\]
Observe that ${\rm WF}_h(\tilde{F}_h(\tau))\subset \bbar{U_{t_0,\delta}}$ uniformly in 
 $\tau$ (as the order functions used to define $\mathbf{A}_h(N_0,\tau)$ are uniform in $\tau$-- see Lemma~\ref{l:escape-function}) and that
 $\|\tilde{F}_{h}(\tau)\|_{H_h^1}\lesssim \|\tilde{f}_{h}(\tau)\|_{H^{2N_0+1}_h}\lesssim h,$ where the involved constants are still uniform 
 for $(\tau,\lambda)$ in the allowed region. From the resolvent bound from Proposition~\ref{p:contofresolvent}, one has, uniformly in $(\tau,\lambda)$, 
 $\|\tilde{u}_h(\tau,\lambda)\|_{L^2}\lesssim h^{-100N_0}$.
 In order to apply Lemma~\ref{Dyzwlemma-Lagrangian}, we just need to verify that ${\rm WF}_h(\tilde{u}_h(\tau,\lambda))\cap U=\emptyset$ uniformly in $(\tau,\lambda)$ thanks to the uniformity of  
 $\mathbf{A}_h(N_0,\tau)$ in $(\tau,\la)$. For that purpose, we fix a family 
 $(B_h)_{0<h\leq 1}\subset \Psi_h^0(\ml{M})$ whose semiclassical wavefront set is contained in $\overline{U}$ and we will verify that~\eqref{uniformKtau} holds. To that aim, we will
also need to use the operator (with $\hat{\chi}$ defined in \eqref{defofchi}) 
\[P^{\chi}_{\tau}(h,\lambda):=P_\tau(h,\lambda)-\hat{\chi} ,\]
and the function
\[\quad \tilde{u}_h^{\chi}(\tau,\la):=P^{\chi}_{\tau}(h,\lambda)^{-1}\tilde{F}_h(\tau)\] 
where we recall that $P^{\chi}_{\tau}(h,\lambda)$ is invertible on $L^2(\ml{M})$ for $\lambda\in\overline{\ml{Z}}$ and that the norm of the inverse $||P^{\chi}_{\tau}(h,\lambda)^{-1}||_{L^2\to L^2}=\ml{O}(h^{-1})$ uniformly for $(\tau,\la)$ in the allowed region. Finally, observe that
\[\tilde{u}_h(\tau,\lambda)=\tilde{u}_h^{\chi}(\tau,\lambda)-P_{\tau}(h,\lambda)^{-1}\hat{\chi}\tilde{u}_h^{\chi}(\tau,\lambda).\]
Hence if we can prove that 
\begin{equation}\label{e:WF-chi}\hat{\chi}\tilde{u}_h^{\chi}(\tau,\lambda)=\ml{O}_{L^2}(h^N)\end{equation} 
for all $N$ uniformly in $(\tau,\lambda)$, then it is equivalent to prove 
the wave front properties for $\tilde{u}_h^{\chi}(\tau,\lambda)$ or for $\tilde{u}_h(\tau,\lambda)$ thanks to the resolvent bound of Proposition~\ref{p:contofresolvent}. The remaining of the proof will be devoted to the proof of the wavefront properties of $\tilde{u}_h^{\chi}(\tau,\lambda)$ and along the way, we will verify that~\eqref{e:WF-chi} holds. Hence, this will give the expected conclusion for $\tilde{u}_h(\tau,\lambda)$. To that aim, we will distinguish several cases depending on the location of the open set $U$ we are considering.
 
 \textbf{The elliptic region.}
 We start with the simplest part of phase space, that is when $\overline{U}$ is contained inside the region where the operator $P_{\tau}(h,\lambda)$ is elliptic: we suppose that $(z,\zeta)\in T_0^*\mc{M}$ does not belong to the cone 
 $$C^{us}(\alpha_1):=\left\{(x,\xi)\in T^*\ml{M}\backslash 0:\alpha_1\|\xi_u+\xi_s\|'\geq\|\xi_0\|'\right\},$$
 for some small $\alpha_1>0$ with the conventions of Section \ref{Sec:invneigh}; here and below, the cones are defined with respect to the Anosov decomposition of the vector field $X_0$. The operator 
 $P_{\tau}(h,\lambda)$ is elliptic outside $C^{us}(\alpha)$ uniformly for $\tau$ small enough. 
 We can then use  the fact that ${\rm WF}_h(B_h)$ is contained in a region where the principal symbol of $P_{\tau}(h,\lambda)$ is uniformly (in $(\tau,\lambda)$) 
 bounded away from $0$. This allows us to write, for every $N\geq 1$,
 $$B_h=\tilde{B}_h^N(\tau,\lambda)P_{\tau}(h,\lambda)+\ml{O}_{L^2\rightarrow L^2}(h^N)$$
where $\tilde{B}_h^N(\tau,\lambda)\in \Psi_h^0(\ml{M})$ and where the constant in the remainder are uniform in $(\tau,\lambda)$ in the allowed region. Note that $\tilde{B}_h^N(\tau,\lambda)$ depends on $(\tau,\lambda)$ but, as 
 these two parameters remain bounded, ${\rm WF}_h(\tilde{B}_h^N(\tau,\lambda))\subset U$ uniformly in $(\tau,\la)$. Gathering these informations, we get 
 $$\|B_h\tilde{u}_h(\tau,\lambda)\|_{L^2} \leq \|\tilde{B}_h^N(\tau,\lambda)\tilde{F}_h(\tau)\|_{L^2}+\ml{O}(h^N)\|\tilde{u}_h(\tau,\lambda)\|_{L^2}.$$
 Since ${\rm WF}_h(\tilde{F}_h(\tau))\subset \bbar{U_{t_0,\delta}}$ (uniformly in $\tau$)  does not intersect $\bbar{U}$,  we find that, for every $N\geq 1$, there exists $C_N>0$ such that, for every $(\tau,\lambda)$ 
 in the allowed region, $\|B_h\tilde{u}_h(\tau,\lambda)\|_{L^2}\leq C_N h^{N-100N_0}$ which is exactly~\eqref{uniformKtau} ($N_0$ being fixed and all the constants depending on a finite number of derivatives of $a$ and $S$).  
 The same ellipticity argument shows that the same property holds with $\tilde{u}_h^\chi(\tau,\lambda)$ replacing $\tilde{u}_h(\tau,\lambda)$. 
 \begin{rem}
 Keeping in mind that we will also need to prove~\eqref{e:WF-chi}, we already make the following obervation. Since $P_\tau^\chi(h,\la)$ is elliptic in $\{||\xi||\leq R\}$ and outside $C^{us}(\alpha_1)$, the same type of ellipticity argument shows that uniformly for $(\tau,\la)$ in the allowed region we have, as $\|\zeta\|\in [2R,4R]$ and as $t_0>0$ is small enough,
\begin{equation}\label{uchiell}
 {\rm WF}_h(\tilde{u}^\chi_h(\tau,\lambda))
\subset  \{\|\xi\|>R\}\cap C^{us}(\alpha_1).
\end{equation}
\end{rem}
It now remains to deal with the part of phase space where the symbol of $P_{\tau}(h,\lambda)$ is not elliptic. 
 
\textbf{The characteristic region away from the strongly unstable cone.} We start with the regularity/smallness away from $E_u^*(X_{\tau_0})$ for large $\|\xi\|$. To that aim, we shall verify that one can find some $0<\alpha_1<\alpha_0$, some $R'>0$ large enough so that, for each $N>0$, for each $(z,\zeta)\in C^{us}(2\alpha_1)$ with $\|\zeta\|\in [2R,4R]$, for each $\tilde{B}_h$ which is microlocalized in $C^{ss}(\alpha_1)\cap\{\|\xi\|\geq R'\}$, there exists some $C_{N,\tilde{B},a,S}>0$ such that for all $\tau$ close enough to $\tau_0$ and $\la\in \mc{Z}$,
 \begin{equation}\label{radialestim}\left\| \tilde{B}_h\tilde{u}_h(\tau,\la)\right\|^2_{L^2}  \leq   C_{N,\tilde{B},a,S}h^{N}, \quad \left\| \tilde{B}_h\tilde{u}^\chi_h(\tau,\la)\right\|^2_{L^2}  \leq  C_{N,\tilde{B},a,S}h^{N},\end{equation}
 with $C_{N,\tilde{B},a,S}$ depending on a finite number of derivatives of $(a,S)$ as in the formulation of Lemma~\ref{Dyzwlemma-Lagrangian}.

We postpone the proof of estimate~\eqref{radialestim} and we first show how to use it in order to conclude when $(z,\zeta)\in C^{us}(2\alpha_1)\setminus C^{uu}(\alpha_1)$. To see this, we first observe that one can find some $T_1>0$ such that
$$\Phi^\tau_{-T_1}(\overline{U})\subset  C^{ss}(\alpha_1)\cap\{\|\xi\|\geq R'\}.$$
Take now $\tilde{B}_h=B_h$. As $(z,\zeta)\in C^{us}(\alpha_1)$ (hence not in the trapped set of the flows $\Phi_t^{\tau}$, given by $E_0^*(X_\tau)$), by taking $U$ and $\delta$ small enough, we can suppose that, for every $t\in[0,T_1]$ 
 and for any $\tau$ small, $\Phi^{\tau}_{-t}(\overline{U})\cap\overline{U_{t_0,\delta}}=\emptyset.$ Hence, by propagation of 
 singularities~\cite[Prop.~2.5]{DyZw13} for the operator $iP_\tau(h,\la)$ and by the regularity estimates~\eqref{radialestim} near the radial source, 
 one knows that $\|B_h\tilde{u}_h(\tau,\la)\|_{L^2}\leq C_Nh^{N}$ for all $N$ with $C_N$ uniform in $(\tau,\la)$ (in the allowed region). 
 Note that due to the compactness of ${\rm WF}_h(B_h)$, evaluating 
 $\|B_hu_h\|_{L^2}$ or $\|B_h\tilde{u}_h\|_{L^2}$ is equivalent. Here, we notice that, due to the facts that we just use propagation for  a uniform finite time and that the Hamiltonian flow $\Phi^\tau_t$ is smooth in $\tau$, the proof of \cite[Prop.~2.5]{DyZw13} can be repeated uniformly for $\tau$ close enough to $0$. Note that the same argument also works for $\tilde{u}_h^{\chi}$ as we can apply propagation of singularities~\cite[Prop.~2.5]{DyZw13} with the operator $iP^{\chi}_{\tau}(h,\lambda)$ as well (using that $\chi_1^2\geq 0$). This concludes the proof of~\eqref{uniformKtau} for $\tilde{u}^\chi_h(\tau,\la)$ and $\tilde{u}_h(\tau,\la)$ away from $C^{uu}(\alpha_1)$, noting one more time that the constants depend on a finite number of derivatives of the functions $(a,S)$ appearing in the definition of $f_h$.

Hence, up to the fact that we still have to prove the radial estimates~\eqref{radialestim}, we are left with the points $(z,\zeta)\in C^{uu}(\alpha_1)$. Note that equation~\eqref{radialestim} gave something slightly stronger than what we need to handle the points away from the strongly unstable cone. Yet, this stronger statement will turn out to be useful below when dealing with the points in the strongly unstable cone $C^{uu}(\alpha_1)$.

\textbf{The strongly unstable region.} We now fix $(z,\zeta)\in C^{uu}(\alpha_1)$ with $\|\zeta\|\in [2R,4R]$. In that case, we will need to use the auxiliary sequence $(\tilde{u}_h^\chi(\tau,\la))_{0<h\leq 1}$. First, we begin with the proof of~\eqref{e:WF-chi}. Recalling the construction of $\hat{\chi}$ and~\eqref{uchiell}, we already know that
$$\left({\rm WF}_h(\tilde{u}_h^\chi(\tau,\la))\cap {\rm WF}_h(\hat{\chi})\right)\subset C^{us}(\alpha_1)\cap \{R\leq\|\xi\|\leq 3R/2\}.$$
We fix some point $(x,\xi)\in C^{us}(2\alpha_1)\setminus C^{uu}(\alpha_1)$ satisfying $\|\xi\|\in [R/2,3R/2]$ and we see similarly that there is a uniform time $T_2>0$ 
such that and for every $\tau$ close enough to $0$, 
$\Phi^\tau_{-T_2}(x,\xi)\in C^{ss}(\alpha_1/2)\cap\{\|\xi\|\geq 2R'\}$. As in the previous step, we can apply propagation of singularities~\cite[Prop.~2.5]{DyZw13} and the radial estimates~\eqref{radialestim} near the stable cone
to $\tilde{u}_h^\chi(\tau,\la)$ with the operator $iP_h^\chi(\tau,\la)$. From that, we deduce that, uniformly in $(\tau,\la)$, ${\rm WF}_h(\tilde{u}_h^\chi(\tau,\la))\cap V=\emptyset$ for $V$ a small neighborhood of 
$(x,\xi)$. Thus, one has, uniformly in $(\tau,\la)$,  
 \begin{equation}\label{uchihorsdeEu}
\left({\rm WF}_h(\tilde{u}_h^\chi(\tau,\la))\cap {\rm WF}_h(\hat{\chi})\right)\subset\left( C^{uu}(\alpha_1)\cap \{\|\xi\|\in [R,3R/2]\}\right).
\end{equation}
If $\alpha_1$ is chosen small enough, then, for each $(x,\xi)\in C^{uu}(\alpha_1)$ 
 with $\|\xi\|\in [R,3R/2]$, there is a uniform time $T_3>0$ (with respect to $\tau$) such that 
 $\Phi^\tau_{-T_3}(x,\xi)\in \{(x,\xi)\in T^*\mc{M};\|\xi\|\leq R/2\}$. We now combine propagation of singularities as above with the elliptic estimate~\eqref{uchiell}. From the above, we conclude that, 
 uniformly in $(\tau,\la)$,
 \begin{equation}\label{WFchiuchi}
 {\rm WF}_h(\tilde{u}_h^\chi(\tau,\la))\cap {\rm WF}_h(\hat{\chi})=\emptyset,
 \end{equation} 
from which we can deduce~\eqref{e:WF-chi} as expected. As already said, we find that $\tilde{u}_h^\chi(\tau,\la)=\tilde{u}_h(\tau,\la)+\ml{O}_{L^2}(h^N)$ uniformly in $(\tau,\la)$ (again all the constants depend on a finite number of derivatives of the functions $(a,S)$ defining $f_h$). Hence, to conclude the proof of the Proposition, it remains to show that, if $B_h$ is microlocalized inside a neighborhood $U$ of
 $(z,\zeta)\in C^{uu}(\alpha_1)$ with $\|\zeta\|\in[2R,4R]$, then $B_h\tilde{u}_h^\chi(\tau,\la)=\ml{O}(h^N)$ uniformly in $(\tau,\lambda)$. For that purpose, it is sufficient to combine 
 propagation of singularities \cite[Prop.~2.5]{DyZw13} with the elliptic estimate~\eqref{uchiell} as before. Indeed, as above and up to shrinking $U$ a little bit, there is $T_4>0$ such that 
 $\Phi^\tau_{-T_4}(\overline{U})\subset \{\|\xi\|\leq R/2\}$ uniformly in $\tau$ and such that $\Phi^{\tau}_{-t}(\overline{U})\cap\overline{U_{t_0,\delta}}=\emptyset$ for every $0\leq t\leq T_4.$

 We have now dealt with every point $(z,\zeta)$ satisfying $\|\zeta\|\in[2R,4R]$. As already explained, combined with Lemma~\ref{Dyzwlemma-Lagrangian}, this concludes the proof of the Proposition except for the estimates~\eqref{radialestim} that still have to be proved.

\textbf{Proof of the radial estimates~\eqref{radialestim}}.
Let us now give the proof of these crucial estimates that were used to handle the points $(z,\zeta)$ in the characteristic region. To that aim, we will make use of the radial propagation estimates from~\cite{Va, DyZw13},  the only difference being that we need to verify the uniformity in the parameter $\tau$.
First of all, we write that, uniformly in $(\tau,\lambda)$,
\begin{equation}\label{e:positive-quantization}
 \forall v\in C^\infty(\ml{M};\ml{E}),\quad \left\|\tilde{B}_hv\right\|^2_{L^2}=\langle \Op_h(b(h))v,v\rangle+\ml{O}(h^{N+1})\|v\|_{L^2}^2,
\end{equation}
 where $b(h)=\sum_{j=0}^Nh^j b_j$ are symbols supported in $C^{ss}(\alpha_1)\cap\{\|\xi\|\geq R'\}$. In particular,
 \begin{equation}\label{e:positive-quantization2}\left\| \tilde{B}_hv\right\|^2_{L^2}  =  {\rm Re}(\langle \Op_h(b_0) v,v\rangle_{L^2})+ 
 h\langle \tilde{R}_h(\tau,\lambda) v,v\rangle_{L^2} +\ml{O}(h^N)\|v\|_{L^2}^2 ,\end{equation}
 where $\tilde{R}_h(\tau,\lambda)\in\Psi_h^0(\ml{M};\ml{E})$ satisfies 
 ${\rm WF}_h(\tilde{R}_h(\tau,\la))\subset \overline{U}$.
 
 We now fix a nondecreasing smooth function $\tilde{\chi}_1$ on $\IR$ which is equal to $1$ on $[20N_0,+\infty)$ and to $0$ 
 on $(-\infty,4N_0]$. Take $\alpha_1<\alpha_0$ small, and using Remark~\ref{r:alternative-order-function} (recall that $N_1=20N_0$) we set
\[\chi_\tau(x,\xi):=\tilde{\chi}_1\left(\tilde{m}_\tau^{N_0,20N_0}(x,\xi)\right).\]
For $\|\xi\|_x\geq 1$, we have $\chi_\tau\equiv 0$ outside $C^{ss}(\alpha_0)$, $\chi_\tau\equiv 1$ on $C^{ss}(\alpha_1)$ and 
$\left\{H_{\tau},\chi_\tau\right\}\leq 0$ on $C^{us}(\alpha_0)$. We will use this smooth function in order to microlocalize our operators 
near $C^{ss}(\alpha_1)$ at infinity (the radial source). 
After possibly adjusting $\alpha_1,R'$ 
and thanks to~\eqref{e:decay-stable-X}, 
we may suppose that there exist $R_0<\tilde{R}_0$ such that $f(x,\xi)\geq \tilde{R}_0$ on $C^{ss}(\alpha_1)\cap\{\|\xi\|\geq R'\}$ and 
$f(x,\xi)\leq R_0$ on $\bbar{U_{t_0,\delta}}$. We fix $\tilde{\chi}_2$ to be a nondecreasing smooth function on $\IR$ which is equal 
to $1$ near $[\ln(1+\tilde{R}_0),+\infty)$ and to $0$ near $(-\infty,\ln(1+R_0)]$. We set
\[\chi_2(x,\xi)=\tilde{\chi}_2(\ln(1+f(x,\xi))).\] 
With these conventions, one has  $\chi_2\equiv 1$ in a neighborhood of $C^{ss}(\alpha_1)\cap\{\|\xi\|\geq R'\}$, $\chi_2\equiv 0$ 
in a neighborhood of $\bbar{U_{t_0,\delta}}$ and $\{H_{\tau},\chi_2\}(x,\xi)\leq 0$ for $\|\xi\|_x\geq 1$ such that $(x,\xi)\in C^{ss}(\alpha_0)$, for all $\tau$ near $\tau_0$. We now define $A_h(\tau)=A_h^*(\tau)$ in $\Psi_h^0(M;\ml{E})$ 
with principal symbol $a_{\tau}:=\chi_{\tau}\chi_2{\rm Id}$ and ${\rm WF}_h(A_h(\tau))\subset \supp(a_\tau)$, thus ${\rm WF}_h(A_h(\tau))\cap \bbar{U_{t_0,\delta}}=\emptyset$ uniformly for $(\tau,\la)$ in the allowed region. 
From the composition rules for pseudo-differential operators, 
 \[\begin{split}
 A_h(\tau)P_{\tau}(h,\lambda)+P_{\tau}(h,\lambda)^*A_h(\tau)=& 
 h\Op_h\left(\left(\left\{H_{\tau},a_{\tau}\right\}-2a_{\tau} \left(\text{Re}(\lambda)-\left\{H_{\tau},G_{\tau}^{N_0}\right\}\right)\right)\text{Id}+a_\tau \mathcal{O}_\tau(1)\right)\\ 
 & +\ml{O}_{\Psi_h^0(\ml{M},\ml{E})}(h^2).\end{split}\]
 Note that the remainder $\mathcal{O}_\tau(1)$ is independent of $N_0$ and that $\ml{O}_{\Psi_h^0(\ml{M},\ml{E})}(h^2)$ has its semiclassical wavefront set contained in $\cup_{\tau}\text{supp}(a_{\tau})$ uniformly in $(\tau,\lambda)$. We can now compare the principal symbol of $h^{-1} \left(A_h(\tau)P_{\tau}(h,\lambda)+P_{\tau}(h,\lambda)^*A_h(\tau)\right)$ with $b_0$: from our construction, one can find some constant $c_{N_0,b_0}>0$ so that
 \[c_{N_0,b_0}b_0\text{Id}\leq \left(-\left\{H_{\tau},a_{\tau}\right\}+2a_{\tau}\left({\rm Re}(\la)-\left\{H_{\tau},G_{\tau}^{N_0}\right\}\right)\right)\text{Id}+a_\tau \mathcal{O}_\tau(1) .\]
Note that we got the negativity of the symbol provided that we choose $N_0$ large enough in a manner that depends only on $b_0$ and $\mc{Z}$
 (recall that 
 $\left\{H_{\tau},G_{\tau}^{N_0}\right\}\leq -c_0N_0$ for every $\|\xi\|_x\geq 1$ when $(x,\xi)\in C^{ss}(\alpha_1)$). We can then use the G\"arding inequality proved in \cite[Proposition E.34]{DyZw4} to this symbol: combining  with~\eqref{e:positive-quantization2}, we get for all $v$ in 
 $C^\infty(\mc{M};\ml{E})$
 \[\left\|\tilde{B}_h v\right\|^2_{L^2}  \leq  -(c_{N_0,b_0}h)^{-1}2{\rm Re}(\langle A_h(\tau)P_{\tau}(h,\lambda)v,v\rangle_{L^2})+ 
 h\langle R_h(\tau,\lambda) v,v\rangle_{L^2} +\ml{O}(h^N)\|v\|_{L^2}^2 ,\]
 where $R_h(\tau,\lambda)\in\Psi_h^0(\ml{M};\ml{E})$ satisfies 
 ${\rm WF}_h(R_h(\tau,\la))\subset \mc{V}$ with $\mc{V}$ a small neighborhood of $\cup_{\tau}\text{supp}(a_{\tau})$ in $\bbar{T^*M}$ uniform in $(\tau,\lambda)$. 
 Then, for all $v$ in  $C^\infty(\mc{M};\ml{E})$ and uniformly in $(\tau,\lambda)$, one has 
 $$\left\| \tilde{B}_hv\right\|^2_{L^2}  \leq  2(c_{N_0,b_0}h)^{-1}\|A_h(\tau)P_{\tau}(h,\lambda)v\|_{L^2}\|v\|_{L^2} + 
 h\langle R_h(\tau,\lambda) v,v\rangle_{L^2} +\ml{O}(h^N)\|v\|_{L^2}^2.$$
 This is a kind of weakened version of the radial estimates (near the source) from~\cite{Va, DyZw13} which holds uniformly in $(\tau,\lambda)$. Using that 
 $\|\tilde{u}_h(\tau,\lambda)\|_{L^2}\leq Ch^{-100N_0}$ uniformly in $(\tau,\lambda)$, 
 we find by letting\footnote{We can use \cite[Lemma E.45]{DyZw4} to justify the convergence in the inequality.} $v\to \tilde{u}_h(\tau,\lambda)$ that, for all $N>0$, there is $C_N>0$ so that
 \[\left\| \tilde{B}_h\tilde{u}_h(\tau,\lambda)\right\|^2_{L^2}  \leq  
C_Nh^{-1-100N_0}\|A_h\tilde{F}_h\|_{L^2}+ 
 h\langle R_h(\tau,\lambda) \tilde{u}_h(\tau,\lambda),\tilde{u}_h(\tau,\lambda)\rangle_{L^2} +C_Nh^{N-200N_0}.\]
Using the facts that ${\rm WF}_h(\tilde{F}_h(\tau))\subset \bbar{U_{t_0,\delta}}$ 
 and ${\rm WF}_h(A_h(\tau))\cap \bbar{U_{t_0,\delta}}=\emptyset$ uniformly in $(\tau,\la)$ 
 we obtain that, for every $N\geq 1$, there exists $C_N>1$ such that
 $\| A_h(\tau)\tilde{F}_h(\tau)\|_{L^2}\leq C_Nh^{N+1}$
uniformly in $(\tau,\lambda)$. Hence, one has, uniformly in $(\tau,\lambda)$,
\[\left\| \tilde{B}_h\tilde{u}_h(\tau,\lambda)\right\|_{L^2}^2\leq h\langle R_h(\tau,\lambda) \tilde{u}_h(\tau,\lambda),\tilde{u}_h(\tau,\lambda)\rangle_{L^2}+C_Nh^{N-200N_0}.\]
We can now reiterate this procedure with $\tilde{B}_h\tilde{B}_h$ replaced by $ h^{\frac{1}{2}}R_h(\tau,\lambda)$ 
which satisfies ${\rm WF}_h(R_h(\tau,\la))\subset \mc{V}$, thus not intersecting $\bbar{U_{t_0,\delta}}$. 
After a finite number of steps, we find 
$\left\|\tilde{B}_h\tilde{u}_h(\tau,\lambda)\right\|_{L^2}\leq C_Nh^{\frac{N}{2}-100N_0}$ uniformly in $(\tau,\lambda)$ which is exactly what we expected (noting that the constants depend on a finite number of derivatives of $a$ and $\chi$). 
The case with $\tilde{u}_h^\chi(\tau,\la)$ is exactly the same by using that 
${\rm WF}_h(\hat{\chi})\cap \{\|\xi\|\geq R'\}\cap C^{us}(\alpha_1)=\emptyset$. Hence, $P_\tau(h,\la)$ 
coincide with $P^\chi_h(\tau,\la)$ microlocally in the region $\{\|\xi\|\geq R'\}\cap C^{us}(\alpha_1)$ where we do the analysis. 
This concludes the proof of~\eqref{radialestim} and thus of Proposition~\ref{p:boundonWF}.

\subsection{Proof of~Theorem \ref{continofFk}}
Using Proposition~\ref{p:contofresolvent}, we can deduce the sequential continuity of $(\tau,\la)\mapsto Q_\tau(\la)$ in 
$\mc{D}^\prime(\mc{M}\x\mc{M};\mc{E}\otimes\mc{E}')$. Then, using~\eqref{p:boundonWF}, we know that the map is also bounded in $\mathcal{D}^\prime_\Gamma$. Both conditions imply the sequential continuity in $\mathcal{D}^\prime_\Gamma$~\cite[condition $(ii)^\prime)$ in definition 8.2.2]{Ho} and with Lemma~\ref{continuityTrb}, we conclude the proof of 2) in Theorem \ref{continofFk}. As $A_{\tau}^{(k)}$ acts as a multiplication operator by smooth functions, the same remains true if we consider $(\tau,\la)\mapsto A_{\tau}^{(k)}Q_\tau(\la)$ and this shows that for every $0\leq k\leq n$ the map
\begin{equation}\label{TrbQ}
(\tau,\lambda)\in[-\delta,\delta]\times\overline{\ml{Z}}
\mapsto \text{Tr}^{\flat}\left(A_{\tau}^{(k)}Q_{\tau}(\la)|_{\mc{E}^k}\right)\in\IC
\end{equation}
is continuous. Finally, by an application of the Cauchy formula and by Proposition~\ref{p:contofresolvent}, 
one can verify that, for every $\tau\in[-\delta,\delta]$ and for every $0\leq k\leq n$,
\[\lambda\in\ml{Z}\mapsto\text{Tr}^{\flat}\left(A_{\tau}^{(k)}Q_{\tau}(\la)|_{\mc{E}^k}\right)\]
is an holomorphic function using Cauchy's formula and the continuity of \eqref{TrbQ}.

Finally, let us remark that the arguments of this section combined with~\cite[\S 4]{DyZw13} also show the following
\begin{prop}\label{p:continuity-zeta} Suppose that $X_0$ is an Anosov vector field and that the representation $\rho_0$ 
induced by the connection $\nabla_0$ is such that $\mathbf{X}_0$ 
has no resonance at $\lambda=0$. Then,  for any continuous family of flat connections $t\in [-1,1]\mapsto \nabla_t$ with corresponding
representation $\rho_t$, the maps
$$X\mapsto\zeta_{X,\rho_0}(0)\quad\text{and}\quad t\in [-1,1]\mapsto\zeta_{X_0,\rho_t}(0)$$ 
are continuous near $X_0$ (resp. $t=0$).
\end{prop}
Note that we only treated the case where $X$ varies. Yet, the same argument holds when we vary the flat connection and when we 
fix $X_0$ as it only modifies $\mathbf{X}_0$ by subprincipal symbols.
Given some flat bundle $\left(E,\nabla\right)$, note that the space of flat connections on $E$ is an affine quadric which can be identified with the solutions 
of the Maurer--Cartan equations~\cite[5.8 p.~74]{GM}
$$\{[\nabla,\Theta]+[\Theta\wedge\Theta]=0  ;  \Theta\in\Omega^1(M,End(E))\},$$ which forms a \emph{quadric} in $\Omega^1(M,End(E))$ since one easily verifies that for every such $\Theta$, $\nabla+\Theta:\Omega^k(M,E)\mapsto \Omega^{k+1}(M,E) $ is flat. The topology is induced from the Fr\'echet topology of $\Omega^1(M,End(E))$.

\section{The Fried conjecture in dimension $3$ and some cases in dimension $5$}

\label{s:example}

\subsection{The kernel of ${\bf X}$ at $\la=0$}

In this section, we will analyze when $0$ is not a resonance for the operator $\mathbf{X}$ of \eqref{defbfX} associated to a
vector field $X\in \mc{A}$. We define
\[C^k:=\ker \left({\bf X}|_{\mc{H}_{h_0}^{m_{N_0,N_1}}(M,\mc{E}^k)}\right)^p, \quad C_0^k:=C^k\cap \ker i_X\]
where $p\geq 1$ is the smallest integer so that $\ker ({\bf X}^{(k)})^p=\ker ({\bf X}^{(k)})^{p+1}$, and
where here we mean the kernel on the anisotropic spaces (for some large enough $N_0$ and $N_1$).
By~\cite[Th.~2.1]{DaRi17c}, the complex
\begin{equation}\label{e:ruelle-complex}0\xrightarrow{d^{\nabla}} C^0\xrightarrow{d^{\nabla}} C^1\xrightarrow{d^{\nabla}} \ldots 
 \xrightarrow{d^{\nabla}} C^n\xrightarrow{d^{\nabla}} 0.
\end{equation}
is quasi--isomorphic to the twisted De Rham complex $(\Omega^{\bullet}(\ml{M},E),d^{\nabla})$ hence the cohomology of~\eqref{e:ruelle-complex} 
coincides with 
the twisted De Rham cohomology.  
We will denote by $H^k(\mc{M};\rho)$ the twisted de Rham cohomology of degree $k$ with $\rho$ the representation associated 
with the flat bundle 
$(E,\nabla)$.

We say that $X\in \mc{A}$ is a \emph{contact Anosov flow} if there is $\alpha\in\Omega^1(\mc{M})$
such that $i_X\alpha=1$, $i_Xd\alpha=0$ and $d\alpha$ is symplectic on $\ker \alpha$.
The dimension of $\mc{M}$ will be denoted $n=2n_0+1$ in that case. 
In particular, one has $\mathcal{L}_X\alpha=0$ and $\mathcal{L}_Xd\alpha=0$, and 
$\mathcal{L}_X\mu=0$ if $\mu=\alpha\wedge d\alpha^{n_0}$.
To begin with, we notice a few commutation relations that will be extensively used. For all $u\in\mc{D}'(\mc{M};\mc{E})$
\begin{equation}\label{commut}
{\bf X}i_{X} u=i_{X}{\bf X}u , \quad {\bf X}(\alpha\wedge u)=\alpha\wedge {\bf X} u, \quad {\bf X}(u\wedge d\alpha)=({\bf X}u)\wedge d\alpha.
\end{equation}
The Koszul complex is naturally associated with our problem
\[0\xrightarrow{i_{X}}C^{2n_0+1}\xrightarrow{i_{X}} C^{2n_0}\xrightarrow{i_{X}}\ldots\xrightarrow{i_{X}} C^1\xrightarrow{i_{X}} 
C^0\xrightarrow{i_{X}}0,\]
and in the contact case there is a dual complex
\[0\xrightarrow{\wedge\alpha}C^{0}\xrightarrow{\wedge\alpha} C^{1}\xrightarrow{\wedge\alpha}\ldots\xrightarrow{\wedge\alpha} C^{2n_0}\xrightarrow{\wedge\alpha} C^{2n_0+1}\xrightarrow{\wedge\alpha}0.\]
\begin{lemm}\label{l:acyclic-koszul} 
For $X\in \mc{A}$, the complex $(C^{\bullet},i_X)$ is acyclic. 
If in addition $X$ is contact with contact form $\alpha$, $(C^{\bullet},\wedge \alpha)$ is acyclic and we have a decomposition~: 
\[\forall 0\leq k\leq 2n_0+1,\quad  C^{k}=(C_0^{k-1}\wedge \alpha) \oplus C_0^k.\]
\end{lemm}
\begin{proof} The spectral projector $\Pi_0$ can be expressed as a contour integral involving the resolvent $(\mathbf{X}+\lambda)^{-1}$ on a small circle around $0$. Then, as $\mathbf{X}$ and $i_X$ commute, one can verify that $\Pi_0$ and $i_X$ also commute. Hence,
	if $u\in  C^{k}\cap \ker (i_{X})=C_0^k$, then $i_{X}\Pi_0\left(\theta\wedge u \right)=\Pi_0\left(\theta(X) u \right)=u$ where $\theta\in \Omega^{1}(\ml{M})$ satisfies $\theta(X)=1$ and $\Pi_0$ is the projector on $C^\bullet$. Thus $(C^{\bullet},i_X)$ is acyclic.
According to~\eqref{commut}, $\alpha\wedge u$ belongs to $C^{k+1}$ whenever $u$ belongs to $C^k$. For $u\in C^k\cap\ker (\wedge\alpha)$, one has $\alpha\wedge(i_Xu)=\alpha(X)u=u$. Hence, 
$(C^{\bullet},\wedge \alpha)$ is acyclic. For $u\in C^k$, we can write 
$u=\alpha\wedge i_Xu+(u-\alpha\wedge i_Xu)$ with $u-\alpha\wedge i_Xu\in C_0^k$, and 
if $u\in C_0^k$ satisfies $\alpha\wedge u=0$, then $u=i_X(\alpha\wedge u)=0$. \end{proof}

From the contact structure, we can also deduce the following duality property:
\begin{lemm}\label{l:duality} 
Suppose that $X\in \mc{A}$ is contact, then for every $0\leq k\leq n_0$,
\[C_0^k \simeq C_0^{2n_0-k}, \quad C^k\simeq C^{2n_0+1-k}.\]
\end{lemm}
\begin{proof}
The bundle $N:=\ker \alpha$ is smooth and $\omega:=d\alpha$ is symplectic on $N$. The form $\omega$ induces 
a non-degenerate pairing $G$ on $\Lambda^kN^*$ for each $k\in[1,2n_0]$, invariant by $X$.
Following~\cite[p.~43]{LM} (see also \cite{Li,Ya}), we can define a (smooth) Hodge star operator $\star: \Lambda^kN^*\to \Lambda^{2n_0-k}N^*$ 
\[ \beta_1\wedge \star \beta_2:= G(\beta_1,\beta_2)\omega^{n_0}/n_0!.\]
$\star$ 
is a bundle isomorphism follows from~\cite[15.2 p.~43]{LM}.
One can check from $L_XG=0$ and $L_X\omega=0$ ($L_X$ the Lie derivative) 
that ${\bf X}\star=\star {\bf X}$, and thus $\star: C_0^k\to C_0^{2n_0-k}$ is an isomorphism since 
$\star\star=\, {\rm Id}$. It remains to use Lemma \ref{l:acyclic-koszul} to obtain $C^k\simeq C^{2n_0+1-k}$.
\end{proof}

\begin{prop}\label{p:vanishing-kernel} 
Suppose that $X\in \mc{A}$ is contact on $\mc{M}$ with dimension 
$2n_0+1$. The following statements are equivalent:
\begin{enumerate}
\item $C^{n_0-1}=0$ and $H^{n_0}(\mc{M},\rho)=0$,
\item $C^{n_0}=0$,
\item  For all $ 0\leq k\leq 2n_0+1$, $C^k=0$.
\end{enumerate}
Suppose that $X\in \mc{A}$ (not necessarily contact) on a $3$-manifold $\mathcal{M}$ and that $X$ preserves some smooth volume form. Then, if $C_0=0$, one has
$$(\Omega^{\bullet}(\ml{M},E),d^{\nabla})\text{ is acyclic}\ \Longleftrightarrow\ \forall 1\leq k\leq 3,C^k=0.$$
\end{prop}
\begin{proof} The statement (3)$\implies$(1) follows from the quasi--isomorphism 
between $\left(C^\bullet,d^\nabla\right)$ and $\left(\Omega^\bullet(\mathcal{M},E),d^\nabla\right)$.  
	Let us show (1)$\implies$(2). Since 
	$C^{n_{0}-1}=0$, we have $C^{n_{0}+2}=0$ by Lemma~\ref{l:duality}. Moreover, by Poincar\'e duality,
	$H^{n_0}(\mc{M},\rho)=H^{n_0+1}(\mc{M},\rho)=0$. Then, still from the quasi--isomorphism, we have that $d^\nabla : 
	C^{n_{0}}\mapsto C^{n_{0}+1}$ is an isomorphism. 
	We can now use the acyclicity of $(C^{\bullet},i_{X})$ and the same argument shows $i_{X}: C^{n_{0}+1}\mapsto 
C^{n_{0}}$ is an isomorphism. So, combined with Lemma~\ref{l:acyclic-koszul}, this shows that ${\bf 
X}_{|C^{n_{0}}}=i_{X}d^\nabla+d^\nabla i_{X}=i_{X}d^\nabla:C^{n_{0}}\mapsto 
C^{n_{0}}$ is an isomorphism. However, by our definition, ${\bf 
X}_{|C^{n_{0}}}$ is nilpotent. Thus, $C^{n_{0}}=C^{n_{0}+1}=0$. To 
show (2)$\implies$(3), from Lemmas~\ref{l:acyclic-koszul} and~\ref{l:duality}, it suffices to 
show that $C^{n_{0}}_{0}=C^{n_{0}-1}_{0}=0$ implies $C^{k}_{0}=0$ for 
every $0\leq k\leq n_{0}-2$. By \cite[Cor.~2.7]{Ya}, $u\mapsto u \wedge (d\alpha)$ maps $C_0^{k}\to C_0^{k+2}$ 
injectively\footnote{This follows from surjectivity of the map 
$u\in C^\infty(\mc{M};\mc{E}_0^{n-k-2})\mapsto u\wedge d\alpha \in C^\infty(\mc{M};\mc{E}_0^{n-k})$.} if $k\leq n_0-1$, 
thus we have $\dim C^{n_{0}}_{0}\geq \dim C^{n_{0}-2}_{0}\geq \ldots$ and 
$\dim C^{n_{0}-1}_{0}\geq \dim C^{n_{0}-3}_{0}\geq \ldots$, which 
shows that (2)$\implies$(3).

In case $n=3$ (i.e., $n_{0}=1$), the proof of the converse sense is 
the same as before. For the direct sense, we cannot use Lemma 
\ref{l:duality}. But we still have $ C^0= C^3=0$ since 
$X$ preserves some smooth volume form $\mu$. The rest of the proof is 
exactly the same as (1)$\implies$(2) given before.
\end{proof}

\begin{lemm}\label{l:DZmonodromylemma}  
Assume $X\in \mc{A}$ preserves a smooth volume form $\mu$ and assume 
$(E,\nabla)$ is a bundle with flat unitary connection.
Let $u$ be an element of $C^0$ such that $\mathbf{X}u=0$.
Then $u\in C^\infty(\ml{M};E)$ and $d^\nabla u=0$.
\end{lemm}
\begin{proof}
Note that ${\bf X}^*=-{\bf X}$ on  $C^\infty(\mc{M};E)$, since ${\bf X}\mu=0$
and that for $v_1,v_2\in C^\infty(\mc{M};E)$,
\[ \cjg {\bf X}v_1,v_2\cjd_{L^2}=
\int_{\mc{M}} \cjg {\bf X}v_1,v_2\cjd_{E}\mu=\int_{\mc{M}} {\bf X}(\cjg v_1,v_2\cjd_{E})\mu-\int_{\mc{M}}\cjg v_1,{\bf X}v_2\cjd_{E}\mu=-\cjg v_1,{\bf X}v_2\cjd_{L^2}.\]
Hence, we can apply \cite[Lemma 2.3]{DyZw} and deduce that $u\in 
C^\infty(\ml{M};E)$. 
Now we use the argument of \cite[Lemma 3]{FRS}. We can lift $u$ to its universal cover $\til{\ml{M}}$
to get a  bounded $\pi_1(\mc{M})$ equivariant $\til{u}\in C^\infty(\til{\mc{M}};\cc^r)$ satisfying 
$\til{u}(\til{\varphi}_t(x))=\til{u}(x)$ for all $x\in \mc{M}$ and $\til{\varphi}_t$ is the lifted flow on $\til{\ml{M}}$. This implies $d\til{u}_{\varphi_{-t}(x)}=(d\til{\varphi}_t)^T_{\varphi_{-t}(x)}d\til{u}_x$. 
For $x\in \mc{M}$ assume that $d\til{u}_x\not\in E_s^*\oplus E_0^*$, then as $t\to +\infty$ we get
$|d\til{u}_{\varphi_{-t}(x)}|_{\cc^r}\to +\infty$, but $|d\til{u}|_{\cc^r}\in L^\infty$ thus a contradiction. 
The same argument by letting $t\to -\infty$ tells us that  
$d\til{u}_x\in E_u^*\oplus E_0^*$ thus $d\til{u}_x\in E_0^*(x)$. But $d\til{u}(X)=0$, thus $d\til{u}(x)=0$.  Then $d^\nabla u=\nabla u=0$ on $\mc{M}$.
\end{proof}

\subsection{Proof of Theorem \ref{t:maintheo2} - The Fried conjecture in dimension~$3$}
We start with the first statement in Theorem \ref{t:maintheo2}. Let $X_0$ be an Anosov vector field preserving a smooth volume form $\mu$ 
and $\nabla$ be a flat unitary connection on a Hermitian bundle $E$ inducing an acyclic representation $\rho$. By 
Lemma~\ref{l:DZmonodromylemma}, we find $C^0=0$ and by Proposition \ref{p:vanishing-kernel}, we obtain 
$C^k=0$ for all $k\in [0,3]$. Then Theorem \ref{t:maintheo} shows that $\zeta_{X,\rho}(0)=\zeta_{X_0,\rho}(0)$ for all 
$X$ in a neighborhood $\mc{U}(X_0)\subset \mc{A}$ of $X_0$. 

Let us show the second part of Theorem \ref{t:maintheo2}. It suffices to show that there is a sequence 
$X_n\in \mc{A}$ such that $X_n\to X_0$ in $C^\infty(\mc{M};T\mc{M})$ and such that $|\zeta_{X_n,\rho}(0)|^{-1}=\tau_\rho(\mc{M})$. 
Sanchez-Morgado~\cite[Th.~1]{Morg96} (based on~\cite{Morg93, Rughanalytic, Friedanalytic}) showed that transitive 
analytic Anosov vector fields $X$ satisfy 
$|\zeta_{X,\rho}(0)|^{-1}=\tau_\rho(\mc{M})$ if there is a closed 
orbit $\gamma$ of $X$ so that $\ker 
(\rho([\gamma])-\varepsilon_{\gamma}^j{\rm Id})=0$ for each $j\in 
\{0,1\}$. 
Among other things including the spectral construction of~\cite{Rughanalytic}, Sanchez-Morgado's argument relied 
crucially on the existence (for Anosov transitive flows on $3$-manifolds) of a Markov partition~\cite[p.~885]{Ra}
whose rectangles have boundaries in $W^u(\gamma)\cup W^s(\gamma)$ for any fixed closed orbit $\gamma$. 
Recall that, for Anosov transitive flows, $W^{u/s}(\gamma) $ is everywhere dense in $\mc{M}$.

If the monodromy property is satisfied for some orbit $\gamma$ of $X_0$, then, for all vector fields $X$ in a small neighborhood
$\mc{U}(X_0)$, there is a periodic orbit $\gamma_X$ of $X$ in the same free homotopy class and 
the corresponding flow is topologically transitive by the strong structural stability Theorem~\ref{t:structural}. Therefore, the results of 
Sanchez-Morgado applies for any $X$ in $\ml{U}(X_0)$ provided that it satisfies some analyticity property.
The conclusion of the proof is then given by the following  when there exists a closed orbit $\gamma$ 
such that the monodromy property of~\cite{Morg96} is verified.
\begin{prop}\label{appanalytic}
There exists a real analytic structure on $\ml{M}$ compatible with the $C^\infty$ structure 
and a sequence $(X_n)_n\subset \mc{A}$ of analytic Anosov vector fields 
such that $X_n\rightarrow X_0$ in the $C^\infty$ topology. 
\end{prop}
\begin{proof}
By Whitney~\cite[Th.~1 p.~654, Lemma 24 p.~668]{Whitney} (see also \cite[Th.~7.1 p.~118]{Hirsch}), there exists a $C^\infty$ embedding $\sigma$
of $\ml{M}$ into $\mathbb{R}^N$ for some $N\in \mathbb{N}$ such that 
$\sigma(\ml{M})$ is a real analytic submanifold of $\mathbb{R}^N$.
It follows from such embedding that the manifold $\ml{M}$ inherits some analytic structure
compatible with the $C^\infty$ structure of $\ml{M}$ since $\ml{M}$ is diffeomorphic to some analytic submanifold of $\mathbb{R}^N$.
The tangent bundle $T\ml{M}\mapsto \ml{M}$ also inherits the real analytic structure from $\ml{M}$ which makes it a real 
analytic bundle in the sense of~\cite[Def.~2.7.8 p.~57]{KrantzParks}. 
Therefore
by the Grauert--Remmert Theorem~\cite[Th.~5.1 p.~65]{Hirsch}, the space of analytic maps
$\ml{M}\mapsto T\ml{M}$ is everywhere dense in $C^\infty(\ml{M},T\ml{M})$ for the strong $C^\infty$-topology. 
In particular,
a vector field $X$ on $\ml{M}$ is understood as a smooth map $\ml{M}\mapsto T\ml{M}$ transverse to the fibers of $T\ml{M}$
which is $C^1$ stable.
Hence any analytic map $\ml{M}\mapsto T\ml{M}$ sufficiently close to $X$ in the $C^1$ topology will be
transverse to the fibers of $T\ml{M}$ and its image in $T\ml{M}$ can be realized as the graph of a real analytic section $\tilde{X}$ of $T\ml{M}$ 
(see also~\cite[Cor.~5.49 p.~106]{ElCieStein} for similar results).
\end{proof}
 It now remains to discuss when we only suppose that $\rho$ is acyclic and that $H^{1}(\ml{M},\IR)\neq\{0\}.$ In that case, one knows 
from~\cite[Th.~2.1]{Pl} that $X_0$ has a closed orbit $\gamma_0$ which is homologically nontrivial. It may happen that no closed orbit verifies 
the monodromy condition of~\cite{Morg96}. Yet, we can fix a closed one form $\alpha_0\in H^1(\ml{M},\IR)$ such that $\int_{\gamma_0}\alpha_0\neq 0$. 
Then, we define $\nabla_s=\nabla+is\alpha_0\wedge$ (with $s\in\IR$) which still induces a unitary representation. Recall that, for $s=0$, $0$ is not a resonance of 
$\mathbf{X}_0$ according to Lemma~\ref{l:DZmonodromylemma} and to Proposition \ref{p:vanishing-kernel}. Thus, for $s$ small enough, $\nabla_s$ also 
remains acyclic thanks to the finite dimensional Hodge theory 
\cite[(1.6)]{BisZhang92} or to \cite[Th.~2.1]{DaRi17c} combined with the fact that $0$ is still not a resonance of $\mathbf{X}_0+is\alpha_0(X_0)$ by the 
arguments\footnote{The proof is even simpler in this case as adding $is\alpha_0(X_0)$ only modifies the operator by a subsprincipal symbol.} used to prove Proposition~\ref{p:contofresolvent}.
One can verify that, for $s\neq 0$ small enough, the monodromy condition of~\cite{Morg96} is verified. 
Hence, for every $s\neq 0$ small enough, one has $|\zeta_{X_0,\rho_s}(0)|^{-1}
=\tau_{\rho_s}(\ml{M}).$ By Proposition~\ref{p:continuity-zeta} and by continuity of the 
map $\rho\mapsto\tau_{\rho}(\ml{M})$, we can conclude that $|\zeta_{X_0,\rho}(0)|^{-1}
=\tau_{\rho}(\ml{M}).$

\subsection{The Fried conjecture near hyperbolic metrics in dimension $n=5$ - Proof of Theorem \ref{t:maintheo3}}
\label{s:hyperbolic}
We refer to \cite{FrInv,BuOl,Ju} for backgrounds on Ruelle/Selberg zeta functions for hyperbolic manifolds.
Let $M=\Gamma\backslash \hh^{n_0+1}$ be a smooth oriented compact $(n_0+1)$-dimensional hyperbolic manifold with $n_0\geq 2$ and $SM=\Gamma\backslash S\hh^{n_0+1}$ its unit tangent bundle, 
where here $\Gamma\subset {\rm SO}(n_0+1,1)$ is a co-compact discrete 
subgroup with no torsion. We consider a unitary representation $\rho: \pi_1(M)\to U(r)$ for $r\in \nn$, and since $\pi_1(SM)\simeq \pi_1(M)$ if 
$n_0+1\geq 3$, 
$\rho$ induces a representation $\til{\rho}:\pi_1(SM)\to U(r)$. By 
considering functions $w$ on $\hh^{n+1}$ with values in $\rr^r$ that are $\Gamma$-equivariant 
(i.e., $\forall \gamma\in \Gamma$, $\gamma^*w=\rho(\gamma)w$), we obtain a rank $r$ vector bundle $E\to M$ equipped with a unitary flat connection $\nabla$, and similarly by using $\til{\rho}$ we obtain a bundle $\til{E}$ and a flat connection $\til{\nabla}$ on $SM$.

We let $X$ be the vector field of the geodesic flow on $\mc{M}:=SM$, and following the previous sections,  
this induces an operator on section of $\til{\mc{E}}:=\oplus_{k}\wedge^kT^*(SM)\otimes \til{E}$
\[ {\bf X}: \Omega(SM;\til{E}) \to \Omega(SM,\til{E}), 
\quad {\bf X}:=i_Xd^{\til{\nabla}}+d^{\til{\nabla}}i_X.\]
and we write ${\bf X}^{(k)}:={\bf X}|_{\Omega_0^k(SM;\til{E})}$ where $\Omega_0^k(SM;\til{E}):=\Omega^k(SM;\til{E})\cap \ker i_X$. 

We define the \emph{dynamical zeta function} of $X$ acting on $\Omega_0^k(SM;\til{E})$ by 
\begin{equation}\label{Zj}
Z_{{\bf X}^{(k)}}(\la)=\exp\Big(-\sum_{\gamma\in \mc{P}}\sum_{j=1}^\infty 
\frac{1}{j}\frac{e^{-\la j \ell(\gamma)}{\rm Tr}(\til{\rho}(\gamma)^j){\rm Tr}(\wedge^k P(\gamma)^j)}{|\det(1-P(\gamma)^j)|}\Big)
\end{equation}
where $\mc{P}$ denotes the set primitive closed geodesics and $P(\gamma)$ is the linearized Poincar\'e map of the geodesic flow along this geodesic. Note that $\mc{P}$ is parametrized by the conjugacy classes of primitive elements in the group $\Gamma$. 
It is known \cite{GLP, DyZw13} that $Z_{{\bf X}^{(k)}}(\la)$ has an analytic continuation 
to $\la\in \cc$ and its zeros are the Ruelle resonances of ${\bf X}^{(k)}$ on $SM$ with multiplicities.

Let $K={\rm SO}(n_0+1)$ be the compact subgroup of $G:={\rm 
SO}(n_0+1,1)$ so that $\hh^{n_0+1}=G/K$ and we can identify 
$S\hh^{n_0+1}=G/H$ where $H:={\rm SO}(n_{0})\subset K$ is the 
stabilizer of a spacelike element in $\rr^{n_0+1,1}$. We have 
$M=\Gamma\backslash G/K$  as locally 
symmetric spaces of rank $1$ and  $SM=\Gamma\backslash G/H$.

Let us define $\xi_p: {\rm SO}(n_0)\to {\rm GL}(S^p\rr^{n_0})$ to be the canonical (unitary) representation of ${\rm SO}(n_0)$ 
into the space $S^p\rr^{n_0}$ of symmetric tensors of order $p$ on $\rr^{n_0}$. 
This representation decomposes into irreducible representations of ${\rm SO}(n_0)$
\[ \xi_p=\sum_{2q\leq p}\sigma_{p-2q}\]
where $\sigma_r:{\rm SO}(n_0)\to {\rm GL}(S_0^r\rr^{n_0})$
is the canonical representation of ${\rm SO}(n_0)$ into the space of trace-free symmetric tensors of order $r$.
 We also define $\nu_l: {\rm SO}(n_0)\to {\rm GL}(\Lambda^l\rr^{n_0})$ to be the canonical (unitary) representation of ${\rm SO}(n_0)$ on $l$-forms.  
 
For each primitive closed geodesic $\gamma$ on $M$ (i.e. primitive closed orbit on $SM$), there is an associated conjugacy class in $\Gamma$, 
with a representative that we still denote by $\gamma\in \Gamma$ and whose axis in $\hh^{n+1}$ descends to the geodesic $\gamma$. There is 
also a neighborhood of  the geodesic in $M$ that is isometric to a neighborhood of the vertical line $\{z=0\}$ in
 the upper half-space $\hh^{n_0+1}=\rr^+_{z_0}\x \rr^{n_0}_z$ quotiented by the elementary group generated by  
\[(z_0,z)\mapsto e^{\ell(\gamma)}(z_0, m(\gamma)z),\]
where $m(\gamma)\in {\rm SO}(n_0)$ and $\ell(\gamma)>0$ being the length of $\gamma$.
The linear Poincar\'e map along this closed geodesic  on 
$E_s\oplus E_u$ is conjugate to the map   
\begin{equation}\label{Pgamma} 
P(\gamma):  (w_s, w_u)\mapsto  (e^{-\ell(\gamma)}m(\gamma) w_s, e^{\ell(\gamma)}m(\gamma) w_u)\end{equation}
where we identify $E_s$ and $E_u$ with $\rr^{n_0}$.

To any irreducible unitary representation $\mu$ of ${\rm SO}(n_0)$ and the representation 
$\rho$ of $\pi_1(M)$ being fixed, we can define a \emph{Selberg zeta function} $Z_{S,\mu}(\la)$ by 
\begin{equation}\label{Selberg} 
Z_{S,\mu}(\la):= \exp\Big( -  \sum_{\gamma\in \mc{P}}\sum_{j=1}^\infty \frac{{\rm Tr}(\tilde{\rho}(\gamma)^j){\rm Tr}(\mu(m(\gamma)^j))e^{-\la j\ell(\gamma)}}{j\det (1-P_s(\gamma)^j)}
\Big)
\end{equation}
where the sum is over all primitive closed geodesics and $P_s(\gamma_0)=P(\gamma_0)|_{E^s}$ 
is the contracting part of $P(\gamma)$. This series converges uniformly for ${\rm Re}(\la)>n_0$. 
For any unitary representation $\mu$ of ${\rm SO}(n_0)$, we can also define $Z_{S,\mu}(\la)$ by the formula \eqref{Selberg}, and 
if $\mu=\sum_{q=1}^{p} \mu_q$ is a decomposition into irreducible representations, $Z_{S,\mu}(\la)=\prod_{q=1}^pZ_{S,\mu_q}(\la)$. 
By \cite[Theorem 3.15]{BuOl}, $Z_{S,\mu}(\la)$ has a meromorphic continuation to $\la\in \cc$, and if $n_0+1$ if odd, the only zeros and poles are contained in ${\rm Re}(\la)\in [0,n_0]$. 

\begin{prop}\label{decompZ_X}
In the region of convergence ${\rm Re}(\la)>n_0$, we have for $k\in [0,n_0]$
\begin{equation}\label{decintoSelberg}
Z_{{\bf X}^{(k)}}(\la)= \prod_{p=0}^\infty\prod_{q=0}^\infty \prod_{l=0}^k 
Z_{S,\nu_l\otimes \nu_{k-l}\otimes \sigma_p}(\la+2(q-l)+p+n_0+k)
\end{equation}
\end{prop}
\begin{proof}
To factorise $Z_{{\bf X}^{(k)}}(\la)$ with some Selberg zeta functions, we compute for $j\in \nn$ 
\[\begin{split}
 |\det(1-P(\gamma)^j)|^{-1}= & e^{-n_0j\ell(\gamma)}\det(1-e^{-j\ell(\gamma)}m(\gamma)^j)^{-1} 
 \det(1-P_s(\gamma)^j)^{-1}\\
 =& e^{-n_0j\ell(\gamma)} \det(1-P_s(\gamma)^j)^{-1}\sum_{r=0}^\infty 
 e^{-rj\ell(\gamma)}\Tr(\xi_r(m(\gamma)^j))
\end{split}  \]
where we used $\det(1-B)^{-1}=\sum_{r=0}^\infty {\rm Tr}(S^rB)$ with $S^rB$ 
the action of $B$ on symmetric tensors on $\rr^{n_0}$ if $B\in {\rm End}(\rr^{n_0})$ with 
$|B|<1$. Now we can use 
\[\begin{split} 
 \sum_{r=0}^\infty 
 e^{-rj\ell(\gamma)}\Tr(\xi_r(m(\gamma)^j))&=
 \sum_{r=0}^\infty \sum_{2q\leq r}
 e^{-rj\ell(\gamma)}\Tr(\sigma_{r-2q}(m(\gamma)^j))\\
 & =\sum_{p=0}^\infty \sum_{q=0}^\infty e^{-(p+2q)j\ell(\gamma)}\Tr(\sigma_p(m(\gamma)^j))
\end{split}\]
Now we also have ${\rm Tr}(\wedge^k P(\gamma)^j)=\sum_{l=0}^ke^{j(2l-k)\ell(\gamma)}
{\rm Tr}(\nu_l(m(\gamma)^j)\otimes \nu_{k-l}(m(\gamma)^j))$.
Combining all this, we thus get 
\[ Z_{{\bf X}^{(k)}}(\la)=\exp\Big(-\sum_{\gamma\in \mc{P}}\sum_{j=1}^\infty 
\sum_{p=0}^\infty \sum_{q=0}^\infty\sum_{l=0}^k
\frac{1}{j}\frac{e^{-(\la+n_0+p+2(q-l)+k)j \ell(\gamma)}{\rm Tr}(\til{\rho}(\gamma)^j){\rm Tr}(\mu_{l,k,p}(m(\gamma)^j)}{|\det(1-P_s(\gamma)^j)|}\Big)\]
with $\mu_{l,k,p}:=\nu_l\otimes \nu_{k-l}\otimes \sigma_p$.
This gives the result. Note that the products in \eqref{decintoSelberg} converge 
for ${\rm Re}(\la)>0$.
\end{proof}
We notice that in each ${\rm Re}(\la)>-N$ for $N>0$ fixed, there is only finitely many Selberg 
type functions in the factorisation \eqref{decintoSelberg} whose exponent of convergence is on the right of $0$, 
this means that only finitely many Selberg terms can bring a zero to $Z_{{\bf X}^{(k)}}(\la)$ in ${\rm Re}(\la)>-N$. 
In particular at $\la=0$, only the terms $l,k,q,p$ with
\begin{equation}\label{condition} 
2(q-l)+p+k\leq 0
\end{equation}
can contribute to a zero (or a pole) there.
Theorem \ref{t:maintheo3} follows directly from Theorem \ref{t:maintheo2}, Fried formula \ref{friedhyp} for hyperbolic manifolds \cite{FrInv} and the following:
\begin{prop}\label{dim3}
Let $M=\Gamma\backslash \hh^{3}$ be a smooth compact oriented hyperbolic manifold and let 
$\rho$ be a unitary representation of $\pi_1(M)$. The multiplicity 
$m_k(0):=\dim C_0^k$ of $0$ as a Ruelle resonance for ${\bf X}^{(k)}$ are given by 
\[\begin{gathered} 
m_0(0)=\dim H^0(M;\rho), \quad  m_1(0)=2\dim H^1(M,\rho), \\ 
m_2(0)=2(\dim H^1(M,\rho)+\dim H^0(M;\rho)), \quad m_{4-k}(0)=m_k(0)
\end{gathered}\]   
where $H^k(M;\rho)$ is the twisted de Rham cohomology of degree $k$ associated to $\rho$.
\end{prop}
\begin{proof}
For $k=0$, from \eqref{decintoSelberg} and \eqref{condition}, we see that 
only the term $Z_{S,\sigma_0}(\la+2)$ can contribute to a zero to the dynamical zeta function 
$Z_{{\bf X}^{(0)}}(\la)$. By Selberg trace formula \cite[Corollary 5.1]{BuOl}, $Z_{S,\sigma_0}(\la+2)$ has a zero of order $\dim \ker \Delta_0$ where $\Delta_0=(d^\nabla)^*d^\nabla$ on sections of the flat Hermitian bundle $(E,\nabla)$ associated to $\rho$.

For $k=1$, the condition \eqref{condition} reduces to the following cases to analyse: $q=0$, $l=1$, $p=0,1$.  For $p=0$, the only term to consider is $Z_{S,\nu_1}(\la+1)$,  
 the Selberg zeta function on $1$-forms. As explained in Section 5.3 of \cite{BuOl}, $\nu_1$ decomposes into two irreducibles $\nu_1^+\oplus \nu_1^-$ and 
by \cite[Proposition 5.6]{BuOl}, each irreducible brings a zero of order 
$-\dim H^0(M,\rho)+\dim H^1(M,\rho)$ at $\la=0$: the contribution to $Z_{{\bf X}^{(1)}}(\la)$ at $\la=0$ coming from $Z_{S,\nu_1}(\la+1)$ is a zero or p\^ole  with order $-2\dim H^0(M,\rho)+2\dim H^1(M,\rho)$.
Next the term $p=1$: we need to look at $Z_{S,\nu_1\otimes \sigma_1}(\la+2)$. 
First we decompose $\sigma_1\otimes \nu_1=\nu_1\otimes \nu_1$ into irreducibles: 
$\nu_1\otimes \nu_1=\sigma_0\oplus\nu_2\oplus\sigma_2$. 
Since $\nu_2\simeq \nu_0$ is equivalent to the trivial representation, 
$Z_{S,\sigma_0\oplus \nu_2}(\la+2)=(Z_{S,\sigma_0}(\la+2))^2$ has a zero of order $2\dim H^0(M,\rho)$ at $\la=0$. 
Now, for $Z_{S,\sigma_2}(\la+2)$ we can use Proposition \ref{selbergtensors}, which gives that the order of $Z_{S,\sigma_2}(\la+2)$ at $\la=0$ is $\dim(\ker \nabla^*\nabla-2)\cap \ker D^*$ where 
$\nabla$ is the twisted covariant derivative on $S^2_0T^*M\otimes E$ and $D^*$ the divergence operator. But by Bochner identity \cite[Equation (2.4)]{DFG}, $\nabla^*\nabla\geq 3$ and thus 
$\dim(\ker \nabla^*\nabla-2)\cap \ker D^*=0$.
We conclude that the order at $\la=0$ of $Z_{{\bf X}^{(1)}}(\la)$ is $2\dim H^1(M,\rho)$.

For $k=2$, if $l=2$ one has to consider $(p,q)=(0,0)$, $(p,q)=(0,1)$, $(p,q)=(1,0)$, $(p,q)=(2,0)$.
First $(p,q)=0$, one get the term $Z_{S,\nu_0}(\la)$ since $\nu_2\simeq \nu_0$, and this has a 
zero of order $\dim H^0(M,\rho)$ at $\la=0$. For 
$(p,q)=(0,1)$,
$Z_{S,\nu_0}(\la+2)$ has a zero of order $\dim H^0(M,\rho)$ at 
$\la=0$. For $(p,q)=(1,0)$, we get the term $Z_{S,\sigma_1}(\la+1)$ 
which has a zero of order $-2\dim H^0(M,\rho)+2\dim H^1(M,\rho)$ as 
discussed above. For $(p,q)=(2,0)$, we get $Z_{S,\sigma_2}(\la+2)$ which has no zero at $\la=0$ as above. Now for 
$l=1$, only $(p,q)=(0,0)$ could contribute, and we get the terms $Z_{S,\nu_1\otimes \nu_1}(\la+2)$ which, as shown above, has a zero of order $2\dim H^{0}(M,\rho)$. This ends the proof.
\end{proof}
\begin{rem}\label{remhyp}
We remark that such a result could alternatively be obtained using the works 
\cite{DFG,KuWe}, with the advantage of knowing the presence of Jordan blocks. The work 
\cite{DFG} also  directly implies that in all dimension $n_0+1\geq 4$, one always has
$m_1(0)=\dim H^1(M;\rho)$ for $M=\Gamma\backslash\hh^{n_0+1}$ co-compact. However, for higher degree forms, and $n_0\geq 4$, it turns out that $m_k(0)$ could a priori be non-topological: 
for example, when $n_0=4$, some computations based on Proposition \ref{decompZ_X} and Selberg formula for irreducible representations as used above shows that when $\dim \ker (\Delta_0-4)=j>0$, these $j$ elements in the kernel contribute to $m_3(0)$. 
\end{rem}
\appendix

\section{Proof of Lemma~\ref{l:escape-function}}\label{a:escape}

\subsection{Family of order functions}\label{ss:order-function}

In this paragraph, we fix the aperture of the cones $\alpha_0>0$ small enough to ensure that 
$C^{ss}(\alpha_0)\cap C^u(\alpha_0)=\emptyset$ and we fix some small parameter $\delta>0$. We construct an order function 
for every $X$ in a small enough neighborhood of $X_0$. For that purpose, we 
closely follow the lines of~\cite[Lemma~2.1]{FaSj}. We fix $T_{\alpha_0}'>1$ $T_{\alpha_0}$ is given by Lemma~\ref{l:time-uniform}. The time $T_{\alpha_0}'$ will be determined later on in a way that depends only on $\alpha_0$. For our construction, 
we also let $m_0(x,\xi)\in\ml{C}^{\infty}(S^*\ml{M},[0,1])$ 
to be equal to $1$ on $C^{u}(\alpha_0)$ and 
to $0$ on $C^{ss}(\alpha_0)$. Then, we set
\begin{equation}\label{e:order-function-X}m_X(x,\xi):=\frac{1}{2T_{\alpha_0}'}\int_{-T_{\alpha_0}'}^{T_{\alpha_0}'}m_0\circ\tilde{\Phi}^{X}_t(x,\xi)dt.\end{equation} 
Note that \emph{$m_X$ depends smoothly on $X$} as we chose $T_{\alpha_0}'$ independently of $X$ near $X_0$.
First of all, we note that
\begin{equation}\label{e:derivative-order-function}\tilde{X}_Hm_X(x,\xi)=\frac{1}{2T_{\alpha_0}'}\left(m_0\circ\tilde{\Phi}^X_{T_{\alpha_0}'}(x,\xi)-
m_0\circ\tilde{\Phi}^X_{-T_{\alpha_0}'}(x,\xi)\right),
\end{equation}
where $\tilde{X}_H$ is the vector field of $\tilde{\Phi}^X_{t}.$ We also observe that, for every $(x,\xi)$ inside $S^*\ml{M}$, the set
$$\ml{I}_{X_0}(x,\xi):=\left\{t\in\IR:\tilde{\Phi}^{X_0}_t(x,\xi)\in S^*\ml{M}\backslash\left(C^u(\alpha_0/2)\cup C^{ss}(\alpha_0/2)\right)\right\}$$
is an interval whose length is bounded by some constant $T_{\alpha_0}''>0$. Fix now a point $(x,\xi)\in S^*\tilde{M}$ and a vector field which is close enough 
to $X_0$ (to be determined). If 
$\tilde{\Phi}^X_t(x,\xi)\in C^u(\alpha_0)$ for every $t\in\IR$, then the set
$$\tilde{\ml{I}}_{X}(x,\xi):=\left\{t\in\IR:\tilde{\Phi}^{X}_t(x,\xi)\in S^*\ml{M}\backslash\left(C^u(\alpha_0)\cup C^{ss}(\alpha_0)\right)\right\}$$
is empty and the same holds if $\tilde{\Phi}^X_t(x,\xi)\in C^{ss}(\alpha_0)$ for every $t\in\IR$. Hence, it remains to bound the length of 
$\tilde{\ml{I}}_{X}(x,\xi)$ when the orbit of $(x,\xi)$ crosses $S^*\ml{M}\backslash\left(C^u(\alpha_0)\cup C^{ss}(\alpha_0)\right)$ and we may 
suppose without loss of generality that $(x,\xi)\in S^*\ml{M}\backslash\left(C^u(\alpha_0)\cup C^{ss}(\alpha_0)\right)$. Up to the fact that we may have to 
decrease a little bit the size of the set $\ml{U}_{\alpha_0}(X_0)$ appearing in Lemma~\ref{l:time-uniform}, we have that $\tilde{\Phi}_{T_{\alpha_0}''}^{X}(x,\xi)$ 
belongs to $C^{u}(\alpha_0)$. Hence, thanks to Lemma~\ref{l:time-uniform}, one finds that, for every $t\geq T_{\alpha_0}''+1$, one has 
$\tilde{\Phi}_t^X(x,\xi)\in C^u(\alpha_0)$. The same holds in backward times. Hence, the diameter of $\tilde{\ml{I}}_{X}(x,\xi)$ is uniformly bounded by 
$2(1+T_{\alpha_0}'')$ and we pick $T_{\alpha_0}'=\frac{1+T_{\alpha_0}''}{\delta}$ for $\delta<1$.

We set
$$\ml{O}^{u}(X)=\tilde{\Phi}^{X}_{T_{\alpha_0}'}(S^*\ml{M}\backslash C^{ss}(\alpha_0))\ \text{and}\ \ml{O}^{ss}(X)=\tilde{\Phi}^{X}_{-T_{\alpha_0}'}
(S^*\ml{M}\backslash C^{u}(\alpha_0)).$$
Let us now discuss the properties of $m_X$ for $X$ belonging to $\ml{U}_{\alpha_0}(X_0)$:
\begin{enumerate}
 \item If $(x,\xi)\in\ml{O}^u(X)$, then $\tilde{\Phi}^X_{-T_{\alpha_0}'}(x,\xi)\notin C^{ss}(\alpha_0)$. Hence, from the definition of $T_{\alpha_0}'$, 
 one has $\tilde{\Phi}^X_{T_{\alpha_0}'}(x,\xi)\in C^{u}(\alpha_0)$ and, from~\eqref{e:derivative-order-function}, one deduce that 
 $\tilde{X}_Hm_X\geq 0$ on $\ml{O}^u(X)$. Similarly, one has
 $$m_X(x,\xi)=\frac{1}{2T_{\alpha_0}'}\left(\int_{-T_{\alpha_0}'}^{-T_{\alpha_0}'+2(T_{\alpha_0}+T_{\alpha_0}'')}m_0\circ\tilde{\Phi}^{X}_t(x,\xi)dt+
 \int_{-T_{\alpha_0}'+2(T_{\alpha_0}+T_{\alpha_0}'')}^{T_{\alpha_0}'}m_0\circ\tilde{\Phi}^{X}_t(x,\xi)dt\right),$$
 from which one can infer
 $$\forall (x,\xi)\in\ml{O}^u(X),\quad m_{X}(x,\xi)\geq1-\frac{T_{\alpha_0}+T_{\alpha_0}''}{T_{\alpha_0}'}=1-\delta.$$
 \item Reasoning along similar lines, one also finds that, for every $(x,\xi)\in\ml{O}^{ss}(X)$, $\tilde{X}_Hm_X\geq 0$ and
 $$m_{X}(x,\xi)\leq \delta.$$
 \item Let $(x,\xi)$ be an element of $S^*\ml{M}\backslash(\ml{O}^u(X)\cup \ml{O}^{ss}(X)).$ In that case, one has 
$\tilde{\Phi}^X_{-T_{\alpha_0}'}(x,\xi)\in C^{ss}(\alpha_0)$ and $\tilde{\Phi}^X_{T_{\alpha_0}'}(x,\xi)\in C^u(\alpha_0)$. Thus, 
one finds \begin{equation}\label{e:decay-order-transit}\tilde{X}_Hm_X(x,\xi)=\frac{1}{2T_{\alpha_0}'}\left(m_0\circ\tilde{\Phi}^X_{T_{\alpha_0}'}(x,\xi)-m_0\circ \tilde{\Phi}^X_{-T_{\alpha_0}'}(x,\xi)\right)
=\frac{1}{2T_{\alpha_0}'}>0.\end{equation}
 \item Let now $(x,\xi)\in S^*\ml{M}\backslash C^{u}(\alpha_0)$. Write
  $$m_X(x,\xi)\leq\frac{1}{2}+\frac{1}{2T_{\alpha_0}'}
 \int_{-T_{\alpha_0}'}^0m_0\circ\tilde{\Phi}^{X}_t(x,\xi)dt\leq\frac{1+\delta}{2}.$$
\end{enumerate}

Let us conclude this construction with the following useful observation:
\begin{lemm}\label{l:inclusion-neighborhoods} Let $\alpha_0>0$ be small enough to ensure that $C^{u}(\alpha_0)\cap C^{ss}(\alpha_0)=\emptyset$. Then, 
there exists $0<\alpha_1<\alpha_0$ and a neighborhood $\ml{U}_{\alpha_0}(X_0)$ of $X_0$ in $\mc{A}$ such that, for every $X\in \mc{U}_{\alpha_0}(X_0)$,
$$ C^u(\alpha_1)\cap S^*\ml{M}\subset \ml{O}^u(X)\quad\text{and}\quad C^{ss}(\alpha_1)\cap S^*\ml{M}\subset \ml{O}^{ss}(X).$$
\end{lemm}

\begin{proof} First of all, we note that by construction of the cones $C^{u}(\alpha)$ (with $\lambda_0>0$ as in Section \ref{Sec:invneigh})
$$\Phi^{X_0}_{T_{\alpha_0}'}\left(C^{u}\left(\alpha_0e^{-\lambda_0 T_{\alpha_0}'}/2\right)\right)\subset C^u(\alpha_0/2).$$
Hence, setting $\alpha_1=\alpha_0e^{-\lambda_0 T_{\alpha_0}'}/2$, one has $\Phi^{X_0}_{T_{\alpha_0}'}\left(C^{u}\left(\alpha_1\right)\right)\subset \overline{C^u(\alpha_0/2)}\subset C^u(\alpha_0).$ By continuity of the maps with respect to $X$ and, up to the fact that we may have to shrink the above neighborhood $\ml{U}_{\alpha_0}(X_0)$ a little bit, one can verify that, for 
every $X\in\ml{U}_{\alpha_0}(X_0)$,
$$C^{u}\left(\alpha_1\right)\subset \Phi^{X}_{-T_{\alpha_0}'}\left(C^u(\alpha_0)\right)\subset \Phi^{X}_{-T_{\alpha_0}'}\left(T^*\ml{M}\backslash C^{ss}(\alpha_0)\right),$$
which concludes the proof by definition of $\mathcal{O}^u(X)$.
\end{proof}

\begin{rem}\label{r:intertwine-time}
 In all the construction so far, we could have defined the cones $C^{uu}(\alpha)$ and $C^s(\alpha)$ (see paragraph~\ref{Sec:invneigh}) and a decaying order function 
 $\tilde{m}_X(x,\xi)$ which is close to $0$ on $C^{s}(\alpha)$ and close to $1$ on $C^{uu}(\alpha)$.
\end{rem}

\subsection{Definition of the escape function}\label{escapefct}

We start with the construction of the function $f(x,\xi)\in \ml{C}^{\infty}(T^*M,\IR_+)$. For $\|\xi\|_x\geq1$, it will be $1$-homogeneous and 
equal to $\|\xi\|_x$ outside the cones $C^{uu}(\tilde{\alpha}_0)$ and $C^{ss}(\tilde{\alpha}_0)$ for $\tilde{\alpha}_0>0$ small enough (to be determined). Following the proof 
of~\cite[Lemma~C.1]{DyZw13} (see also~\cite[Lemma~2.2]{BoWe17}), we set, for $(x,\xi)$ near $C^{ss}(\tilde{\alpha}_0/2)$ and $\|\xi\|_x\geq 1$,
\[f(x,\xi):=\exp\left(\frac{1}{T_1}\int_{0}^{T_1}\ln\|(d\varphi^{X_0}_{t}(x)^T)^{-1}\xi\|_{\varphi_{X_0}^t(x)}dt\right).\]
Recall that, for every $\xi$ in $E_s^*(X_0,x)$, one has $\|(d\varphi^{X_0}_{t}(x)^T)^{-1}\xi\|\leq Ce^{-\beta t}\|\xi\|$ for every $t\geq 0$ (where $C,\beta$ 
are some uniform constants). Hence, if we set $T_1=2\frac{\ln C}{\beta}$, we find that, for every $(x,\xi)\in E_s^*(X_0)$ with $\|\xi\|_x\geq 1$, 
$X_{H_0}f(x,\xi)\leq - f(x,\xi)\frac{\beta}{2}.$ Similarly, picking $T_1$ large enough, we set, for $(x,\xi)$ near $C^{uu}(\tilde{\alpha}_0/2)$ and $\|\xi\|_x\geq 1$,
$$f(x,\xi):=\exp\left(\frac{1}{T_1}\int_{0}^{T_1}\ln\|(d\varphi^{X_0}_{t}(x)^T)^{-1}\xi\|_{\varphi_{X_0}^t(x)}dt\right),$$
and we find that $X_{H_0}f(x,\xi)\geq  f(x,\xi)\frac{\beta}{2}$ on $E_u^*(X_0)$. By continuity, we find that there exists some (small enough) $\tilde{\alpha}_0>0$ 
such that, for every $\|\xi\|_x\geq 1$,
\begin{equation}\label{e:decay-stable-X0}
 (x,\xi)\in C^{ss}(\tilde{\alpha}_0/2)\Rightarrow X_{H_0}f(x,\xi)\leq - f(x,\xi)\frac{\beta}{3},
\end{equation}
and
\begin{equation}\label{e:decay-unstable-X0}
 (x,\xi)\in C^{uu}(\tilde{\alpha}_0/2)\Rightarrow X_{H_0}f(x,\xi)\geq f(x,\xi)\frac{\beta}{3}.
\end{equation}
As the function $f(x,\xi)$ is $1$-homogeneous, we can find a neighborhood $\ml{U}(X_0)$ of $X_0$ in the $\ml{C}^{\infty}$-topology such that, for every 
$X$ in $\ml{U}(X_0)$ and for every $\|\xi\|_x\geq 1$,
\begin{equation}\label{e:decay-stable-X}
 (x,\xi)\in C^{ss}(\tilde{\alpha}_0/2)\Rightarrow X_{H}f(x,\xi)\leq - f(x,\xi)\frac{\beta}{4},
\end{equation}
and
\begin{equation}\label{e:decay-unstable-X}
 (x,\xi)\in C^{uu}(\tilde{\alpha}_0/2)\Rightarrow X_{H}f(x,\xi)\geq f(x,\xi)\frac{\beta}{4}.
\end{equation}
Finally, we note that there exists some uniform constant $C>0$ such that, for every $X$ in $\ml{U}(X_0)$ and for $\|\xi\|_x\geq 1$, 
\begin{equation}\label{e:decay-norm-far}
 -Cf(x,\xi)\leq X_Hf(x,\xi)\leq C f(x,\xi)
\end{equation}
We are now ready to construct our family of escape functions $G_X^{N_0,N_1}(x,\xi)$:
\[G_X^{N_0,N_1}(x,\xi):=m_{X}^{N_0,N_1}(x,\xi)\ln(1+f(x,\xi)),\]
with $m_{X}^{N_0,N_1}\in\ml{C}^{\infty}(T^*M,[-2N_0,2N_1])$ which is $0$-homogeneous for $\|\xi\|_x\geq 1$. In order to construct this 
function, we will make use of the order functions defined in paragraph~\ref{ss:order-function} as in~\cite[p.~337-8]{FaSj}. Before doing that, let 
us observe that
\begin{equation}\label{e:derivative-escape-function}
 X_HG_X^{N_0,N_1}(x,\xi)=X_H(m_{X}^{N_0,N_1})(x,\xi)\ln(1+f(x,\xi))+m_{X}^{N_0,N_1}(x,\xi)\frac{X_Hf(x,\xi)}{1+f(x,\xi)}.
\end{equation}
We now fix a small enough neighborhood $\ml{U}(X_0)$ of $X_0$ so that $f$ enjoys  \eqref{e:decay-stable-X} and~\eqref{e:decay-unstable-X} for all $X$ in $\ml{U}(X_0)$ and so that we can apply the results of 
paragraph~\ref{ss:order-function}. Following~\cite{FaSj}, we set, for $\|\xi\|_x\geq 1$,
\begin{equation}\label{e:order-function}m_X^{N_0,N_1}(x,\xi):=
 N_1\left(2-m_X\left(x,\frac{\xi}{\|\xi\|_x}\right)-\tilde{m}_X\left(x,\frac{\xi}{\|\xi\|_x}\right)\right)-2N_0\tilde{m}_X\left(x,\frac{\xi}{\|\xi\|_x}\right),
\end{equation}
where we used the conventions of paragraph~\ref{ss:order-function} and Remark~\ref{r:intertwine-time}. First, notice that, by construction, $X_H(m_{X}^{N_0,N_1})\leq 0$ for $\|\xi\|_x\geq 1$. 
Recall that the order functions 
$m_X$ and $\tilde{m}_X$ depends on the parameters $\alpha_0>0$ and $\delta>0$ and that they depend smoothly on $X$. 
Now, we fix $0<\delta<\frac{1}{2}\min\{1,\min\{N_0,N_1\}/(N_0+N_1)\}$, $0<16N_0<N_1$ and 
$0<\alpha_0<\tilde{\alpha}_0/2$. We then find that 
$m_X^{N_0,N_1}(x,\xi/\|\xi\|_x)\geq N_1$ on $\ml{O}^{ss}(X)$ and $m_X^{N_0,N_1}(x,\xi/\|\xi\|_x)\leq -N_0$ on $\ml{O}^{uu}(X)$. 
We also have that $m_X^{N_0,N_1}(x,\xi/\|\xi\|_x)\geq\frac{N_1}{4}-2N_0\geq  N_1/8$ for 
$(x,\xi)$ outside $C^{uu}(\alpha_0)$ (as $N_1>16N_0$). We now fix $\alpha_1$ to be the aperture of the cone appearing in Lemma~\ref{l:inclusion-neighborhoods}. This allows to verify the first three requirements of $m_X^{N_0,N_1}$.

\begin{rem}\label{r:alternative-order-function} We could also have defined
$$\tilde{m}_X^{N_0,N_1}(x,\xi):=
 N_1\left(1-m_X\left(x,\frac{\xi}{\|\xi\|_x}\right)\right)-N_0\tilde{m}_X\left(x,\frac{\xi}{\|\xi\|_x}\right).$$
 We still have $\tilde{m}_X^{N_0,N_1}(x,\xi)\geq N_1$ on $\ml{O}^{ss}(X)$, $\tilde{m}_X^{N_0,N_1}(x,\xi)\leq\frac{N_1}{4}-N_0$ outside $C^{ss}(\alpha_0)$.
\end{rem}

Finally, combining $X_H(m_{X}^{N_0,N_1})\leq 0$ with~\eqref{e:derivative-escape-function} for $||\xi||\geq 1$, we immediately get the upper 
bound~\eqref{e:bound-derivative-escape-function}. It now remains to verify the decay property~\eqref{e:decay-escape-function}. For that purpose, we shall use the conventions of paragraph~\ref{ss:order-function} and set, for every $X\in\ml{U}(X_0)$, $$\tilde{\ml{O}}^{uu}(X)=\ml{O}^{uu}(X)\cap\ml{O}^u(X),\ \tilde{\ml{O}}^{0}(X)=\ml{O}^{s}(X)\cap\ml{O}^u(X),\ \text{and}\ \tilde{\ml{O}}^{ss}(X)=\ml{O}^{ss}(X)\cap\ml{O}^s(X),$$
which contains respectively $C^{uu}(\alpha_1)$, $C^u(\alpha_1)\cap C^s(\alpha_1)$ and $C^{ss}(\alpha_1)$ for $\alpha_1>0$ small enough (see Lemma~\ref{l:inclusion-neighborhoods}). Note also that $\tilde{\ml{O}}^{0}(X)$ is contained inside $C^u(\alpha_0)\cap C^s(\alpha_0)$ which is a small vicinity of $E_0^*(X_0)$. Based on~\eqref{e:derivative-escape-function}, we can now establish~\eqref{e:decay-escape-function} except in this small cone around the flow 
direction. Outside $\tilde{\ml{O}}^{uu}(X)\cup \tilde{\ml{O}}^{0}(X)\cup \tilde{\ml{O}}^{ss}(X)$, it follows from~\eqref{e:decay-order-transit} 
and~\eqref{e:derivative-escape-function}. Inside $\tilde{\ml{O}}^{uu}(X)$ and~$\tilde{\ml{O}}^{ss}(X)$, it follows 
from~\eqref{e:decay-stable-X},~\eqref{e:decay-unstable-X} and~\eqref{e:derivative-escape-function}.

\section{Selberg zeta function on trace-free symmetric tensors}
\label{a:selberg}

\begin{prop}\label{selbergtensors}
Let $n$ be even and $M=\Gamma\backslash \hh^{n+1}$ be a compact hyperbolic manifold. Let $\rho:\pi_1(M)\to U(V_\rho)$ be a finite dimensional unitary representation and  
let $\sigma_m$ be the irreducible unitary representation of ${\rm SO}(n)$ into the space $S_0^m\rr^n$ of trace-free symmetric tensors of order $m\geq 1$ on $\rr^n$. Then the Selberg zeta function $Z_{S,\sigma_m}(s)$ on $M$ associated to $\sigma_p$ and $\rho$ is holomorphic and the order of its zeros are given by 
\[{\rm ord}_{s_0}Z_{S,\sigma_m}(s) =\left\{\begin{array}{ll}
\dim\ker (\nabla^*\nabla-n^2/4-m+(s_0-n/2)^2)\cap \ker D^* & \textrm{ if }s_0\not=n/2\\
2\dim\ker (\nabla^*\nabla-n^2/4-m)\cap \ker D^* & \textrm{ if }s_0=n/2
\end{array}\right.\]
where $\nabla$ is the twisted Levi-Civita covariant derivative on $S_0^mT^*M\otimes E$, $E\to M$ being the flat bundle over $M$ obtained from 
the representation $\rho$, and $D^*=-{\rm Tr}\circ \nabla$ is the divergence operator. 
\end{prop}
\begin{proof}
We follow \cite[Theorem 3.15]{BuOl}. First we need to view $\sigma_m$ as the restriction of a sum of irreducibles representations of ${\rm SO}(n+1)$ as in Section 1.1.2 \cite{BuOl}: it is not difficult to check that 
\[ \sigma_m=(\Sigma_m-\Sigma_{m-1})|_{{\rm SO}(n)}\]
where $\Sigma_m$ denotes the irreducible unitary representation of ${\rm SO}(n+1)$ into the space $S_0^m\rr^{n+1}$. By Section 1.1.3 of \cite{BuOl}, there is a $\zz^2$-graded homogeneous vector bundle $V_{\sigma_m}=V_{\Sigma_m}^+\oplus V_{\Sigma_m}^-$ over $\hh^{n+1}$ with 
$V_{\Sigma_m}^+=S_0^{m}\rr^{n+1}$ and $V_{\Sigma_m}^-=S_0^{m-1}\rr^{n+1}$,
and we define the bundle $V_{M,\rho\otimes \sigma_m}=\Gamma\backslash (V_\rho\otimes V_{\sigma_m})$ over $M$. Denoting $E\to M$ the bundle over $M$ obtained from $V_\rho$ by quotienting by $\Gamma$ and $S^m_0T^*M$ the bundle of trace-free symmetric tensors of order $m$ on $M$, the bundle $V_{M,\rho\otimes \sigma_m}$ is isomorphic to the 
bundle $\mc{E}:=(S^m_0T^*M\oplus S^{m-1}_0T^*M)\otimes E$. There is a differential operator 
$A^2_{\sigma_m}$ on $\mc{E}$ constructed from the Casimir operator that has eigenvalues in correspondence with the zeros/poles of $Z_{S,\sigma_m}(s)$, it is given $A^2_{\sigma_m}=-\Omega-c(\sigma_m)$ where $\Omega$ is the Casimir operator and 
$c(\sigma)=n^2/4-|\mu(\sigma_m)|^2-2\mu(\sigma).\rho_{\rm so(n)}$ with
$\mu(\sigma_m)$ the highest weight of $\sigma$ and $\rho_{\rm so(n)}=(\ndemi-1,\ndemi-2,\dots,0)$. Here we have $\mu(\sigma_m)=(m,0,\dots,0)$ thus 
\[c(\sigma_m)=\frac{n^2}{4}-m(m+n-2).\]
We then obtain the formula 
\[ A^2_{\sigma_m}=(\Delta_m-c(\sigma_m))\oplus( \Delta_{m-1}-c(\sigma_m)) \]
where $\Delta_m=\nabla^*\nabla-m(m+n-1)$ is the Lichnerowicz Laplacian on (twisted) trace-free symmetric tensors of order $m$ on $M$ (see for instance \cite[Section 5]{Ha}). Now we have by \cite[Lemma 5.2]{Ha} that $D^*\Delta_m=\Delta_{m-1}D^*$ if $D^*$ is the divergence operator 
defined by $D^*u=-{\Tr}(\nabla u)$, and whose adjoint is $D=\mc{S}\nabla$ is the symmetrised 
covariant derivative. This gives $\Delta_mD=D\Delta_{m-1}$, but since $D$ is elliptic with no kernel by \cite[Proposition 6.6]{HMS}, it has closed range and $D$ gives an isomorphism  
\[ D: \ker (\Delta_{m-1}-c(\sigma_m)-s)\to \ker (\Delta_{m}-c(\sigma_m)-s)\cap (\ker D^*)^\perp\]
for each $s\in\rr$. In particular, one obtains that for each $s\in\rr$
\[ \dim \ker (\Delta_{m}-c(\sigma_m)-s)-\dim \ker (\Delta_{m-1}-c(\sigma_m)-s)=\dim( \ker (\Delta_{m}-c(\sigma_m)-s)\cap \ker D^*).\] 
Now by \cite[Theorem 3.15]{BuOl}, the function $Z_{S,\sigma_m}(s)$ has a zero at $s$ of order 
\[\begin{gathered} 
2 \dim(\dim( \ker (\Delta_{m}-c(\sigma_m)\cap \ker D^*)) \textrm{ if }s=\ndemi\\
\dim(\dim( \ker (\Delta_{m}-c(\sigma_m)\cap \ker D^*)) \textrm{ if }s\not=\ndemi.
\end{gathered}\]
\end{proof}

\end{document}